\documentclass[a4paper]{article}

\usepackage[english]{babel}
\usepackage[utf8x]{inputenc}
\usepackage[T1]{fontenc}
\usepackage{upgreek}
\usepackage{gensymb}
\makeatletter
\newcommand{\ssymbol}[1]{^{\@fnsymbol{#1}}}
\makeatother

\usepackage{amssymb,amsmath,amsfonts,amsthm,latexsym,enumerate,url,cases}
\usepackage{graphicx}
\usepackage[colorinlistoftodos]{todonotes}
\usepackage[colorlinks=true, allcolors=blue]{hyperref}

\usepackage[a4paper,top=3cm,bottom=2cm,left=3cm,right=3cm,marginparwidth=1.75cm]{geometry}

\usepackage{authblk}

\title{Unstable manifolds for rough evolution equations\thanks{This work is supported in part by a NSFC Grant No. 12171084 and the Fundamental Research Funds for the Central Universities No. 2242022R10013.}}
\author{Hongyan Ma}
\author[$\ssymbol{2}$]{Hongjun Gao}
\affil[$\ssymbol{2}$]{  School of Mathematical Sciences, Nanjing Normal University, 210023, Nanjing, China}
\affil[$\ssymbol{2}$]{  School of mathematics, Southeast University, 211189,  Nanjing, China}
\usepackage{fullpage}
\usepackage{amssymb,amsmath,amsfonts,amsthm,eufrak,latexsym,enumerate,url,cases}
\usepackage{mathrsfs}
\numberwithin{equation}{section}
\usepackage{hyperref}
\usepackage{indentfirst}
\usepackage{hyperref}
\hypersetup{colorlinks=true,citecolor=blue,linkcolor=blue,urlcolor=blue}

\newtheorem{theorem}{Theorem}[section] %
\newtheorem{lemma}{Lemma}[section] %
\newtheorem{corollary}{Corollary}[section] %
\newtheorem{definition}{Definition}[section] %
\newtheorem{remark}{Remark}[section] %

\usepackage{cite}
\usepackage{hyperref}
\hypersetup{hypertex=true,
            colorlinks=true,
            linkcolor=blue,
            anchorcolor=blue,
            citecolor=blue}
\usepackage{bm}
\usepackage{authblk}
\usepackage{indentfirst}
\begin{document}

\maketitle

\begin{abstract}
In this paper, we consider a class of evolution equations driven by finite-dimensional $\gamma$-H\"{o}lder rough paths, where $\gamma\in(1/3,1/2]$. We prove the global-in-time solutions of rough evolution equations(REEs) in a sutiable space, also obtain that the solutions generate random dynamical systems. Meanwhile, we derive the existence of local unstable manifolds for such equations by a properly discretized Lyapunov-Perron method.
\end{abstract}

{\bf 2020 Mathematics Subject Classification:} 60H15, 60H05, 37H10, 37D10

{\bf Keywords}: Rough evolution equations, Random dynamical system, Unstable manifolds, Lyapunov-Perron method.
\section{Introduction}
Invariant manifolds are one of the cornerstones of nonlinear dynamical systems and have been widely studied in deterministic systems. However, in many applications, nonlinear dynamical systems are affected by noise. Invariant manifolds have been researched for stochastic ordinary differential equations(SDEs)(see \cite{MR1723992}, \cite{MR1022628}, \cite{MR1178953} et al) and stochastic partial differential equations(SPDEs) (see Chen et al \cite{MR3376184},\cite{MR2016614}, \cite{MR2110052}, \cite{MR2593602} et al ). A key difficulty in studying invariant manifolds of a stochastic partial differential equation is to prove that it generates a stochastic dynamic system. As we all known that a large class of partial differential equations with stationary random coefficients and It\^{o} stochastic ordinary differential equations generate random dynamical systems (see Arnold \cite{MR1723992}). Nevertheless, for stochastic partial differential equations driven by the standard Brownian motion, it is unknown that how to obtain random dynamical systems. The reasons are: (\romannumeral1) the stochastic integral is only defined almost surely where the exceptional set may depend on the initial state; (\romannumeral2) Kolmogorov's theorem is only true for finite dimensional random fields. However, there are some results for additive and linear multiplicative noise(see \cite{MR2016614},\cite{MR2110052}, \cite{MR3289240} et al).

A way to obtain a random dynamical system from a stochastic differential equation is that this equation is driven by $\gamma$-H\"{o}lder continuous path. In this sense, there are two techniques of defining the stochastic integral that are in pathwise sense. For $\gamma>1/2$, these integrals are consistent with the well-known Young integral (see Young \cite{MR1555421} and Z\"{a}hle \cite{MR1640795}). One of the techniques is to define the integral based on fractional derivatives.  There are already some investigations which have proven that the (pathwise) solutions driven by fractional Brownian motion with $\gamma>1/2$ generate random dynamic systems, obtained that the existence of random attractors and invariant manifolds that describe the longtime behaviors of the solutions (see Chen et al \cite{MR3072986}, Gao et al \cite{MR3226746}, Garrido-Atienza et al \cite{MR2660869}, \cite{MR2593602}). For $1/3<\gamma\leq1/2$, Garrido-Atienza et al \cite{MR3479690} have obtained random dynamical systems for stochastic evolution equations driven by multiplicative fractional Brownian noise; Another one is to interpret integral in the rough path sense. Rough path theory (see \cite{MR4174393}, \cite{MR4040992}, \cite{MR2091358} and \cite{MR1654527}) is close to deterministic analytical methods, Bailleul \cite{MR3330818} analyzed flows driven by rough paths and Bailleul et al \cite{MR3624539} studied random dynamical systems for rough differential equations. Kuehn and Neam\c{t}u \cite{MR4284415} have proven the existence and regularity of local center manifolds for rough differential equations by means of a suitably discretized Lyapunov-Perron-type method. Gubinelli and Tindel \cite{MR2599193} generalised theory of rough paths to solve not only
SDEs but also SPDEs: evolution equations driven by the infinite dimensional Gaussian process. Gerasimovi\v{c}s and Hairer \cite{MR4040992} have developed a robust pathwise local solution theory for a class semilinear SPDEs with multiplicative noise driven by a finite dimensional Wiener process. Hesse and Neam\c{t}u \cite{MR4001062}, \cite{MR4097587} have investigated local, global mild solutions and random dynamical systems for rough partial differential equations. Recently, Hesse and Neam\c{t}u \cite{MR11111} have obtained global-in-time solutions and random dynamical systems for semilinear parabolic rough partial differential equations and Neam\c{t}u and Kuehn \cite{MR11112} have derived the centre manifolds for rough partial differential equations base on the work of \cite{MR11111}.

 However, so far, there are little works relate to unstable manifolds of rough evolution equations. Therefore, in this paper, base on \cite{MR3376184}, \cite{MR2593602}, \cite{MR4040992}, \cite{MR11111}, \cite{MR11112} and \cite{MR4284415}, we are going to study random dynamical systems and local unstable manifolds for (2.2). In order to overcome the obstacle that how to obtain random dynamical system of SPDEs with nonlinear multiplicative noise, similar to \cite{MR11111} and \cite{MR4284415}, we choose a proper  space that is different from [16] and [22], give a simpler proof of local solution for rough evolution equations than [16] and obtain the global solutions, also, we obtain random dynamical systems by using the rough integral developed in \cite{MR4040992} and rough path cocycle \cite{MR3624539}. Meanwhile, we obtain the contraction properties of the Lyapunov-Perron map by using rough path estimates, moreover,  by using properly discretized Lyapunov-Perron method, we derive the existence of local unstable manifolds for these equations.

 This paper is structured as follows. In section 2 we provide background on mildly controlled rough paths and study the global solutions of rough evolution equations. Section 3 is devoted to dynamics of rough evolution equations. In section 4 we derive the existence of local unstable manifolds which are based on a discrete-time Lyapunov-Perron method. Since we work with pathwise integrals, so, at each step, it is necessary to control the norms of the random input on a fixed time-interval. By deriving suitable estimates of the mildly controlled rough integrals, the unstable manifold theory is obtained by employing a random dynamical systems approach. The results obtained for the discrete Lyapunov-Perron map can then be extended to the time-continuous one (further details please refer to \cite{MR2593602}, \cite{MR4284415}).
%due to technical reasons as discussed in the following, we let $T\in(0,1]$.

\section{Rough evolution equations}
Throughout this paper, let $T>0$, we consider a separable Hilbert space $\mathcal{H}$ and $A$ is a generator of analytic $C_{0}$-semigroup $\{S_{t}:t\geq0\}$ on the interpolation space $(\mathcal{H}_{\alpha}=\mbox{Dom}(-A)^{\alpha};\alpha\in\mathbb{R})$. We will use the following fact that for all $\alpha\geq\beta$, ${\gamma\in[0,1]}$ and $u\in{\mathcal{H}}_{\beta}$, one has

\begin{equation}\label{2.1}
\begin{aligned}
 \|S_{t}u\|_{\mathcal{H}_{\alpha}} \leq C_{\beta}t^{\beta-\alpha}\|u\|_{\mathcal{H}_{\beta}},\qquad\|S_{t}u-u\|_{\mathcal{H}_{\beta-\gamma}}\leq C_{\gamma}t^{\gamma}\|u\|_{\mathcal{H}_{\beta}}
\end{aligned}
\end{equation}
uniformly over ${t\in(0,T]}$. For an introduction to semigroup theory, one can refer to \cite{MR710486}.

$\mathbf{Notation}$: We denote $\mathcal{H}^{d}_{\alpha}:=\mathcal{L}(\mathbb{R}^{d},\mathcal{H}_{\alpha})\left(\mathcal{H}^{d\times d}_{\alpha}:=\mathcal{L}(\mathbb{R}^{d}\otimes\mathbb{R}^{d},\mathcal{H}_{\alpha})\right)$ as the space of continuously linear operators from $\mathbb{R}^{d}(\mathbb{R}^{d}\otimes\mathbb{R}^{d})$ to $\mathcal{H}_{\alpha}$. For some fixed $\alpha,\beta\in\mathbb{R}$ and $k\in\mathbb{N}$, we denote $\mathcal{C}^{k}_{\alpha,\beta}(\mathcal{H},\mathcal{H}^{n})$ as the space of $k$-order continuously Fr\'{e}chet-differentiable functions $g:{\mathcal{H}}_{\theta}\rightarrow{\mathcal{H}}^{n}_{\theta+\beta}$ for any $\theta\geq\alpha$, $n\in\mathbb{N}$ with bounded derivatives $D^{i}g$, for all $i=1,\cdot\cdot\cdot,k$. Furthermore, we denote $\mathcal{C}_{\alpha,\beta}(\mathcal{H},\mathcal{H})$ as the space of continuous functions $f:{\mathcal{H}}_{\theta}\rightarrow{\mathcal{H}}_{\theta+\beta}$ for any $\theta\geq\alpha$. $\mathcal{C}_{n}([0,T];V)$ as the space of continuous functions from $\Delta_{n}$ to $V$ where $\Delta_{n}:=\{(t_{1},\cdot\cdot\cdot,t_{n}):T\geq t_{1}\geq \cdot\cdot\cdot \geq t_{n}\geq0\}$ for $n\geq1$ and, for notational simplicity, denote $\mathcal{C}([0,T];V)=\mathcal{C}_{1}([0,T];V)$. $C$ stands for a universal constant which may vary from line to line, the dependence of this constant $C=C_{\cdot,\cdot,\cdot\cdot\cdot}$ on certain parameters will be explicitly stated in subscripts.

In this article, we will consider rough evolution equations
 \begin{equation}\label{2.2}
  \left\{
   \begin{aligned}
     dy_{u}=&\left(Ay_{u}+f(y_{u})\right)du+g(y_{u})d\mathbf{w}_{u},u\in[0,T],\\
     y(0)= &\xi\in\mathcal{H},
  \end{aligned}
 \right.
 \end{equation}
 where we assume:
 \begin{itemize}
   \item $f\in\mathcal{C}_{-2\gamma,0}(\mathcal{H},\mathcal{H})$ is global Lipschitz continuous,
   \item $g\in\mathcal{C}^{3}_{-2\gamma,0}(\mathcal{H},\mathcal{H}^{d})$ and such that $\|g(0)\|_{\mathcal{H}^{d}_{\theta}}=C_{0}$ for $\theta\geq-2\gamma$,
   \item $\mathbf{w}$ is the $\gamma$-H\"{o}lder rough path with $\gamma\in(\frac{1}{3},\frac{1}{2}]$ that will be defined as below.
 \end{itemize}
 The mild solution of \eqref{2.2} can be given by
 \begin{equation}\label{2.3}
\begin{aligned}
 y_{t}=S_{t}\xi+\int^{t}_{0}S_{tu}f(y_{u})du+\int^{t}_{0}S_{tu}g(y_{u})d\mathbf{w}_{u},
\end{aligned}
\end{equation}
where the last integral is rough integral, which is pathwise, will be defined below. From now on, for notational simplicity, we denote $S_{ts}:=S_{t-s}$ for $0\leq s<t\leq T$. In this section, we will prove the global in time solution of \eqref{2.2} and its truncated equation, this is essential for one to consider the invariant manifolds for rough evolution equation.

 First of all, we review some concepts and results on rough path theory, for more details, please refer to  \cite{MR4174393} and \cite{MR4040992}.
 Given a Banach space $V$ endowed with the norm $\|\cdot\|_{V}$, for $h\in\mathcal{C}([0,T];V)$, $p\in\mathcal{C}_{2}([0,T];V)$, let
 $$\delta h_{t,s}=h_{t}-h_{s},\qquad\delta p_{t,u,s}=p_{t,s}-p_{t,u}-p_{u,s}.$$
 $$\hat{\delta} h_{t,s}=h_{t}-S_{ts}h_{s},\qquad\hat{\delta} p_{t,u,s}=p_{t,s}-p_{t,u}-S_{tu}p_{u,s}.$$
Notice that $V$ is one of the spaces in which the action of the semigroup $S$ makes sense.
Then, for $0<\gamma<1$ we set
  $$|h|_{\gamma,V}=\sup_{s,t\in[0,T]}\frac{\|\delta h_{t,s}\|_{V}}{|t-s|^{\gamma}},\qquad\|h\|_{\gamma,V}=\sup_{s,t\in[0,T]}\frac{\|\hat{\delta}h_{t,s}\|_{V}}{|t-s|^{\gamma}} ,\qquad|p|_{\gamma,V}=\sup_{s,t\in[0,T]}\frac{\|p_{t,s}\|_{V}}{|t-s|^{\gamma}}.$$
 Consequently, one can define the spaces as below:
  $$\mathcal{C}^{\gamma}([0,T];V)=\{h\in \mathcal{C}([0,T];V):|h|_{\gamma,V}<\infty\},$$
  $$\mathcal{C}^{\gamma}_{2}([0,T];V)=\{p\in \mathcal{C}_{2}([0,T];V):|p|_{\gamma,V}<\infty\},$$
  $$\hat{\mathcal{C}}^{\gamma}([0,T];V)=\{h\in \mathcal{C}([0,T];V):\|h\|_{\gamma,V}<\infty\}.$$
%\end{definition}
\begin{remark}\label{remark1}
Since the semigroup $S$ is not H\"{o}lder continuous at $t=0$, hence, from now on, we will choose $\hat{\delta}$ operator and $\hat{\mathcal{C}}^{\gamma}$ type H\"{o}lder spaces for our evolution setting to overcome this obstacle.
\end{remark}
In addition, we endow $\mathcal{C}([0,T];V)$ with the supremum norm $\|h\|_{\infty,V}=\sup_{0\leq t\leq T}\|h_{t}\|_{V}$. For notational simplicity, in the cases of $V=\mathcal{H}_{\alpha}$, $\mathcal{H}^{d}_{\alpha}$ and $\mathcal{H}^{d\times d}_{\alpha}$, we will denote
$|h|_{V}=|h|_{\gamma,\alpha}$,
$\|h\|_{\gamma,V}=\|h\|_{\gamma,\alpha}$, $\|h\|_{\infty,V}=\|h\|_{\infty,\alpha}$.
\begin{definition}\label{Definition2.1}
For $\gamma\in(\frac{1}{3},\frac{1}{2}]$, we define the space of $\gamma$-H\"{o}lder rough paths(over $\mathbb{R}^{d}$) as those pairs $\mathbf{w}=(w,w^{2})\in\mathcal{C}^{\gamma}([0,T];\mathbb{R}^{d})\times\mathcal{C}^{2\gamma}_{2}([0,T];\mathbb{R}^{d}\otimes \mathbb{R}^{d})$ satisfy the Chen's relation, i.e. if for $s\leq u\leq t\in[0,T]$
$$w^{2}_{t,s}-w^{2}_{t,u}-w^{2}_{u,s}=\delta w_{u,s}\otimes \delta w_{t,u}.$$
This space is denoted as $\mathscr{C}^{\gamma}([0,T];\mathbb{R}^{d})$. For two rough paths $\mathbf{w}=(w,w^{2}),\mathbf{\tilde{w}}=(\tilde{w},\tilde{w}^{2})\in\mathscr{C}^{\gamma}([0,T];\mathbb{R}^{d})$, we define the rough metric $\varrho_{\gamma}$ as:
$$\varrho_{\gamma}(\mathbf{w},\mathbf{\tilde{w}})=|w-\tilde{w}|_{\gamma}+|w^{2}-\tilde{w}^{2}|_{2\gamma}.$$
\end{definition}
\begin{definition}\label{Definition2.2}
Let $\mathbf{w}\in\mathscr{C}^{\gamma}([0,T];\mathbb{R}^{d})$, for some $\gamma\in(\frac{1}{3},\frac{1}{2}]$, we call $(y,y^{'})\in\mathcal{\hat{C}}^{\gamma}\left([0,T];\mathcal{H}_{\alpha}\right)\times\mathcal{\hat{C}}^{\gamma}\left([0,T];\mathcal{H}^{ d}_{\alpha}\right)$ a mildly controlled rough path, if the remainder term $R^{y}$ is defined by
\begin{equation}\label{2.4}
\begin{aligned}
R^{y}_{t,s}=\hat{\delta}y_{t,s}-S_{ts}y^{'}_{s}\delta w_{t,s},\quad\mbox{for}\quad s\leq t\in[0,T],
\end{aligned}
\end{equation}
which belongs to $\mathcal{C}^{2\gamma}_{2}\left([0,T];\mathcal{H}_{\alpha}\right)$, then we call $y^{'}$ mildly Gubinelli derivative of $y$ and denote $(y,y^{'})\in\mathscr{D}^{2\gamma}_{S,w}\left([0,T];\mathcal{H}_{\alpha}\right)$.
\end{definition}
Notice that, when one replaces $\mathcal{H}_{\alpha}$ by $\mathcal{H}^{ d}_{\alpha}$, the above definition is also true.
Meanwhile, a seminorm on this space is defined as
$$\|y,y^{'}\|_{w,2\gamma,\alpha}=\|y^{'}\|_{\gamma,\alpha}+|R^{y}|_{2\gamma,\alpha}.$$
The norm of $\mathscr{D}^{2\gamma}_{S,w}\left([0,T];\mathcal{H}_{\alpha}\right)$ is defined as
$$\|y,y^{'}\|_{\mathscr{D}^{2\gamma}_{S,w}}=\|y_{0}\|_{\mathcal{H}_{\alpha}}+\|y^{'}_{0}\|_{\mathcal{H}^{ d}_{\alpha}}+\|y,y^{'}\|_{w,2\gamma,\alpha}.$$
\begin{remark}\label{remark6}
\cite{MR11112} have used controlled rough path given in \cite{MR11111} which is different from the one we use. Here we incorporate semigroup into the definition of controlled rough path as in \cite{MR4040992}.
\end{remark}
 According to \eqref{2.4}, one can easily derive that
\begin{equation}\label{2.5}
\begin{aligned}
\|y\|_{\gamma,\alpha}&\leq|R^{y}|_{2\gamma,\alpha}T^{\gamma}+\|y^{'}\|_{\infty,\alpha}|w|_{\gamma}\leq (1+|w|_{\gamma})(\|y^{'}_{0}\|_{\mathcal{H}^{d}_{\alpha}}+\|y,y^{'}\|_{w,2\gamma,\alpha}T^{\gamma}).
\end{aligned}
\end{equation}
Furthermore, given a mildly controlled rough path, one can define the rough integral as below:
\begin{theorem}\label{Theorem2.1}
Let $T>0$ and $\mathbf{w}\in\mathscr{C}^{\gamma}([0,T];\mathbb{R}^{d})$ for some $\gamma\in(\frac{1}{3},\frac{1}{2}]$. Let $(y,y^{'})\in\mathscr{D}^{2\gamma}_{S,w}([0,T];\mathcal{H}^{d}_{\alpha})$. Furthermore, $\mathcal{P}$ stands for a partition of $[0,T]$. Then the integral defined as
\begin{equation}\label{2.6}
\begin{aligned}
\int^{t}_{s}S_{tu}y_{u}d\mathbf{w}_{u}:=\lim\limits_{|\mathcal{P}|\rightarrow0}\sum_{[u,v]\in\mathcal{P}}S_{tu}(y_{u}\delta w_{v,u}+y^{'}_{u}w^{2}_{v,u})
\end{aligned}
\end{equation}
exists as an element of $\mathcal{\hat{C}}^{\gamma}([0,T];\mathcal{H}_{\alpha})$ and satisfies that for every $0\leq\beta<3\gamma$ we have
\begin{equation}\label{2.7}
\begin{aligned}
&\left\|\int^{t}_{s}S_{tu}y_{u}d\mathbf{w}_{u}-S_{ts}y_{s}\delta w_{t,s}-S_{ts}y^{'}_{s}w^{2}_{t,s}\right\|_{\mathcal{H}_{\alpha+\beta}}\\
&\lesssim(|R^{y}|_{2\gamma,\alpha}|w|_{\gamma}+\|y^{'}\|_{\gamma,\alpha}|w^{2}|_{2\gamma})|t-s|^{3\gamma-\beta}.
\end{aligned}
\end{equation}
Moreover, the map
$$(y,y^{'})\rightarrow(z,z^{'}):=\left(\int^{\cdot}_{0}S_{\cdot u}y_{u}d\mathbf{w}_{u},y\right)$$
is a continuous linear map from $\mathscr{D}^{2\gamma}_{S,w}([0,T];\mathcal{H}^{d}_{\alpha})$ to $\mathscr{D}^{2\gamma}_{S,w}([0,T];\mathcal{H}_{\alpha})$.
%Also, we have the bound
%\begin{equation}
%\begin{aligned}
%\|z,z^{'}\|_{w,2\gamma,\alpha}\lesssim\|Y\|_{\gamma,\alpha}+(\|Y^{'}_{0}\|_{\mathcal{H}^{d\times d}_{\alpha}}+\|y,y^{'}\|_{w,2\gamma,\alpha})(|w|_{\gamma}+|w^{2}|_{2\gamma}),
%\end{aligned}
%\end{equation}
here the underlying constant depends on $\gamma$, $d$ and $T$ and can be chosen uniformly over $T\in(0,1]$.
\end{theorem}
  In our case, one need to consider a suitable class of nonlinearities integrands, base on Lemma 3.14 of \cite{MR4040992}, we consider mildly controlled rough path compose with regular functions as follows, since the proof is identical to the one of Lemma 3.7 of \cite{MR4040992}, we omit it here.
\begin{lemma}\label{Lemma2.1}
Let $g\in \mathcal{C}^{2}_{\alpha,0}(\mathcal{H},\mathcal{H}^{d})$, $T>0$ and $(y,y^{'})\in\mathscr{D}^{2\gamma}_{S,w}([0,T];\mathcal{H}_{\alpha})$, for some $\mathbf{w}\in\mathscr{C}^{\gamma}([0,T];\mathbb{R}^{d})$, $\gamma\in(1/3,1/2]$. Moreover, suppose $y\in \mathcal{\hat{C}}^{\eta}([0,T];\mathcal{H}_{\alpha+2\gamma})$, $\eta\in[0,1]$ and $y^{'}\in L^{\infty}([0,T];\mathcal{H}^{ d}_{\alpha+2\gamma})$. Define $(z_{t},z^{'}_{t})=(g(y_{t}),Dg(y_{t})y^{'}_{t})$, then, $(z,z^{'})\in\mathscr{D}^{2\gamma}_{S,w}([0,T];\mathcal{H}^{d}_{\alpha})$ and satisfies the bound
\begin{equation}\label{2.8}
\begin{aligned}
&\|z,z^{'}\|_{w,2\gamma,\alpha}\\
&\leq C_{g,T}(1+|w|_{\gamma})^{2}(1+\|y^{'}_{0}\|_{\mathcal{H}^{d}_{\alpha}}+\|y,y^{'}\|_{w,2\gamma,\alpha})\\
&\quad\cdot(1+\|y_{0}\|_{\mathcal{H}_{\alpha+2\gamma}}+\|y^{'}_{0}\|_{\mathcal{H}^{d}_{\alpha}}+\|y\|_{\eta,\alpha+2\gamma}+\|y^{'}\|_{\infty,\alpha+2\gamma}+\|y,y^{'}\|_{w,2\gamma,\alpha}).
\end{aligned}
\end{equation}
The constant $C_{g,T}$ depends on $g$ and the bounds of its derivatives, meanwhile, it depends on time $T$, but can be chosen uniformly over $T\in(0,1]$.
\end{lemma}
According to Lemma \ref{Lemma2.1}, the composition with regular functions needs greater spatial regularity conditions for mildly controlled rough path. Hence, in our evolution setting, in order to study the global in time solutions of \eqref{2.2} in a suitable space, as in \cite{MR4040992}, we need the following space:
%\begin{definition}[\cite{MR4040992}]\label{definition4}
%Let $\mathbf{w}\in\mathscr{C}^{\gamma}([0,T];\mathbb{R}^{d})$, for some $\gamma\in(\frac{1}{3},\frac{1}{2}]$, then for any $\beta\in\mathbb{R}$ and $\eta\in[0,1]$, we define a space
$$ \mathscr{D}^{2\gamma,\beta,\eta}_{S,w}([0,T]);\mathcal{H}_{\alpha})=\mathscr{D}^{2\gamma}_{S,w}([0,T];\mathcal{H}_{\alpha})\cap(\mathcal{\hat{C}}^{\eta}([0,T];\mathcal{H}_{\alpha+\beta})\times L^{\infty}([0,T];\mathcal{H}^{d}_{\alpha+\beta})),
$$
where $\beta\in\mathbb{R}$ and $\eta\in[0,1]$.
%\end{definition}
 Let $(y,y^{'})\in\mathscr{D}^{2\gamma,\beta,\eta}_{S,w}([0,T]);\mathcal{H}_{\alpha})$, the seminorm of this space is defined as:
$$\|y,y^{'}\|_{w,2\gamma,\beta,\eta}=\|y\|_{\eta,\alpha+\beta}+\|y^{'}\|_{\infty,\alpha+\beta}+\|y,y^{'}\|_{w,2\gamma,\alpha}.$$
The norm of this space is defined as below:
$$\|y,y^{'}\|_{\mathscr{D}^{2\gamma,\beta,\eta}_{w}}=\|y_{0}\|_{\mathcal{H}_{\alpha+\beta}}+\|y^{'}_{0}\|_{\mathcal{H}^{d}_{\alpha}}+\|y\|_{\eta,\alpha+\beta}+\|y^{'}\|_{\infty,\alpha+\beta}+\|y,y^{'}\|_{w,2\gamma,\alpha}.$$
Moreover, we will denote $\mathcal{\hat{C}}^{0}=\mathcal{C}$ for $\eta=0$.

%Since, from (2.7) we obtain easily that
%\begin{small}
%\begin{equation}
%\begin{aligned}
%\|\hat{\delta}Z_{t,s}\|_{\mathcal{H}_{\alpha+2\gamma}}&=\left\|\int^{t}_{0}S_{t-u}Y_{u}d\mathbf{W}_{u}-S_{t-s}\int^{s}_{0}S_{s-u}Y_{u}d\mathbf{W}_{u}\right\|_{\mathcal{H}_{\alpha+2\gamma}}\nonumber\\
%&=\left\|\int^{t}_{s}S_{t-u}Y_{u}d\mathbf{W}_{u}\right\|_{\mathcal{H}_{\alpha+2\gamma}}\\
%&\leq\left\|\int^{t}_{s}S_{t-u}Y_{u}d\mathbf{W}_{u}-S_{t-s}Y_{s}\delta W_{t,s}-S_{t-s}Y^{'}_{s}\mathbb{W}_{t,s}\right\|_{\mathcal{H}_{\alpha+2\gamma}}\\
%&\quad+\|S_{t-s}Y_{s}\delta W_{t,s}\|_{\mathcal{H}_{\alpha+2\gamma}}+\|S_{t-s}Y^{'}_{s}\mathbb{W}_{t,s}\|_{\mathcal{H}_{\alpha+2\gamma}}\\
%&\lesssim(|R^{Y}|_{2\gamma,\alpha}|W|_{\gamma}+\|Y^{'}\|_{\gamma,\alpha}|\mathbb{W}|_{2\gamma})|t-s|^{3\gamma-2\gamma}\\
%&\quad+\|Y\|_{\infty,\alpha+2\gamma}|W|_{\gamma}|t-s|^{\gamma}+\|Y^{'}\|_{\infty,\alpha+2\gamma}|\mathbb{W}|_{2\gamma}|t-s|^{2\gamma},
%\end{aligned}
%\end{equation}
%\end{small}
%then, for $0\leq\eta\leq\gamma$, we have
%$$\|Z\|_{\eta,\alpha+2\gamma}\lesssim(|W|_{\gamma}+|\mathbb{W}|_{2\gamma})(\|Y_{0}\|_{\mathcal{H}_{\alpha+2\gamma}}+\|Y\|_{\eta,\alpha+2\gamma}+\|Y^{'}\|_{\infty,\alpha+2\gamma}+\|Y,Y^{'}\|_{W,2\gamma,\alpha}),$$
%where the underlying constant depends on $\gamma$, $d$ and $T$ and is uniform over $T\in(0,1]$.

Furthermore, from Lemma 2.1, we know that composition with regular functions maps $\mathscr{D}^{2\gamma,2\gamma,\eta}_{S,w}([0,T];\mathcal{H}_{\alpha})$ to $\mathscr{D}^{2\gamma,2\gamma,0}_{S,w}([0,T];\mathcal{H}^{d}_{\alpha})$, for $\eta\in[0,1]$.
  For notational simplicity, we denote
$$\mathcal{D}^{2\gamma,\eta}_{w}([0,T];\mathcal{H}_{\alpha}):=\mathscr{D}^{2\gamma,2\gamma,\eta}_{S,w}([0,T];\mathcal{H}_{\alpha-2\gamma}),\qquad 0\leq\eta<\gamma,$$
the seminorm and norm of $\mathcal{D}^{2\gamma,\eta}_{w}([0,T];\mathcal{H}_{\alpha})$ are respectively denoted as
$\|\cdot,\cdot\|_{w,2\gamma,2\gamma,\eta}$ and $\|\cdot,\cdot\|_{\mathcal{D}^{2\gamma,\eta}_{w}}$.
\begin{remark}\label{remark11}
   Notice that, as in \cite{MR4040992}, for notational simplicity, the authors have denoted $\mathcal{D}^{2\gamma}_{w}([0,T];\mathcal{H}_{\alpha}):=\mathscr{D}^{2\gamma,2\gamma,\gamma}_{S,w}([0,T];\mathcal{H}_{\alpha-2\gamma})$ and considered the solution in $\mathcal{D}^{2\gamma}_{w}([0,T];\mathcal{H})$ which is differ from our case. In our situation, in order to facilitate the study of the global in time solution of \eqref{2.2}, we will choose to consider \eqref{2.2} in the space $\mathcal{D}^{2\gamma,\eta}_{w}([0,T];\mathcal{H})$ which is bigger than the space $\mathcal{D}^{2\gamma}_{w}([0,T];\mathcal{H})$ of \cite{MR4040992}.
\end{remark}
%\begin{proposition}[\cite{MR4040992}]\label{proposition1}
%For $\gamma\in(\frac{1}{3},\frac{1}{2}]$, $\mathscr{D}^{2\gamma,2\gamma,0}_{W}([0,T];\mathcal{H}_{-2\gamma})=\mathscr{D}^{2\gamma,2\gamma,0}_{S,W}([0,T];\mathcal{H}_{-2\gamma}).$
%\end{proposition}

%From \cite{MR4040992} Proposition 3.9 and \cite{MR4174393} Corollary 7.3, combining with Lemma 2.1, we easily obtain the result which ensure the reasonable of lemma 2.3 as below:
%\begin{corollary}\label{corollary1}
%Let $(y,y^{'})$, $(v,v^{'})\in\mathcal{D}^{2\gamma,\eta}_{w}([0,T];\mathcal{H})$, we have $(y-v,(y-v)^{'})\in\mathcal{D}^{2\gamma,\eta}_{w}([0,T];\mathcal{H})$ with mildly Gubinelli derivative $(y-v)^{'}=y^{'}-v^{'}$ and $(yv,(yv)^{'})\in\mathscr{D}^{2\gamma,2\gamma,0}_{S,w}([0,T]);\mathcal{H}_{-2\gamma})$ with mildly Gubinelli derivative $(yv)^{'}=y^{'}v+yv^{'}$.
%\end{corollary}
%According to Theorem 2.1, Definition 2.4 and Lemma 2.1, for $(y,y^{'})$,  $(v,v^{'})\in\mathcal{D}^{2\gamma,\eta}_{w}([0,T];\mathcal{H})$ and $g\in \mathcal{C}^{3}_{-2\gamma,0}(\mathcal{H}, \mathcal{H}^{d})$. Assume there exists $M>0$ such that $|w|_{\gamma}$, $|w^{2}|_{2\gamma}$, $\|y,y^{'}\|_{\mathcal{D}^{2\gamma,\eta}_{w}}$ and $\|v,v^{'}\|_{\mathcal{D}^{2\gamma,\eta}_{w}}\leq M$, we have the following lemmas.

\begin{lemma}\label{Lemma2.2}
Let $T>0$, $g\in\mathcal{ C}^{3}_{-2\gamma,0}(\mathcal{H}, \mathcal{H}^{d})$, $(y,y^{'})\in\mathcal{D}^{2\gamma,\eta}_{w}([0,T];\mathcal{H})$, for some $\mathbf{w}\in\mathscr{C}^{\gamma}([0,T];\mathbb{R}^{d})$ with $\gamma\in(\frac{1}{3},\frac{1}{2}]$. We have
$$\left(\int^{\cdot}_{0}S_{\cdot u}g(y_{u})d\mathbf{w}_{u},g(y)\right)\in\mathcal{D}^{2\gamma,\eta}_{w}([0,T];\mathcal{H})$$
and
\begin{small}
\begin{equation}\label{2.9}
\begin{aligned}
\left\|\int^{\cdot}_{0}S_{\cdot u}g(y_{u})d\mathbf{w}_{u},g(y)\right\|_{\mathcal{D}^{2\gamma,\eta}_{w}}
&\leq C_{\gamma,d,T}(1+|w|_{\gamma}+|w^{2}|_{2\gamma})\|g(y),(g(y))^{'}\|_{\mathscr{D}^{2\gamma,2\gamma,0}_{S,w}},
\end{aligned}
\end{equation}
\end{small}
where the constant $C_{\gamma,d,T}$ depends on $\gamma$, $d$ and $T$ and can be chosen uniformly over $T\in(0,1]$.
\end{lemma}

\begin{proof}
 According to \eqref{2.1} and \eqref{2.7} we obtain that
\begin{small}
\begin{equation}
\begin{aligned}\nonumber
\left\|R^{\int^{\cdot}_{0}S_{\cdot u}g(y_{u})d\mathbf{w}_{u}}_{t,s}\right\|_{\mathcal{H}_{-2\gamma}}
&\leq\left\|\int^{t}_{s}S_{tu}g(y_{u})d\mathbf{w}_{u}-S_{ts}g(y_{s})\delta w_{t,s}-S_{ts}(g(y_{s}))^{'}w^{2}_{t,s}\right\|_{\mathcal{H}_{-2\gamma}}\\
&\quad+\|S_{ts}(g(y_{s}))^{'}w^{2}_{t,s}\|_{\mathcal{H}_{-2\gamma}}\\
&\lesssim\left(|R^{g(y)}|_{2\gamma,-2\gamma}|w|_{\gamma}+\|(g(y))^{'}\|_{\gamma,\mathcal{H}^{d\times d}_{-2\gamma}}|w^{2}|_{2\gamma}\right)|t-s|^{3\gamma}\\
&\quad+\|(g(y_{s}))^{'}\|_{\mathcal{H}^{d\times d}_{-2\gamma}}|w^{2}|_{2\gamma}|t-s|^{2\gamma},
\end{aligned}
\end{equation}
\end{small}
then we have
\begin{small}
\begin{equation}
\begin{aligned}\nonumber
\left\|R^{\int^{\cdot}_{0}S_{\cdot u}g(y_{u})d\mathbf{w}_{u}}\right\|_{2\gamma,-2\gamma}
&\lesssim T^{\gamma}(|w|_{\gamma}+|w^{2}|_{2\gamma})\|g(y),(g(y))^{'}\|_{w,2\gamma,-2\gamma}\\
&\quad+|w^{2}|_{2\gamma}\|(g(y))^{'}\|_{\infty,-2\gamma}.
\end{aligned}
\end{equation}
\end{small}
Similarly, we have
\begin{small}
\begin{equation}
\begin{aligned}\nonumber
&\left\|\int^{t}_{s}S_{tu}g(y_{u})d\mathbf{w}_{u}\right\|_{\mathcal{H}}\\
&\leq\left\|\int^{t}_{s}S_{tu}g(y_{u})d\mathbf{w}_{u}-S_{ts}g(y_{s})\delta w_{t,s}-S_{ts}(g(y_{s}))^{'}w^{2}_{t,s}\right\|_{\mathcal{H}}\\
&\quad+\|S_{ts}g(y_{s})\delta w_{t,s}\|_{\mathcal{H}}+\|S_{ts}(g(y_{s}))^{'}w^{2}_{t,s}\|_{\mathcal{H}}\\
&\lesssim\left(|R^{g(y)}|_{2\gamma,-2\gamma}|w|_{\gamma}+\|(g(y))^{'}\|_{\gamma,-2\gamma}|w^{2}|_{2\gamma}\right)|t-s|^{\gamma}\\
&\quad+\|g(y_{s})\|_{\mathcal{H}^{d}}|w|_{\gamma}|t-s|^{\gamma}+\|(g(y_{s}))^{'}\|_{\mathcal{H}^{d\times d}}|w^{2}|_{2\gamma}|t-s|^{2\gamma},
\end{aligned}
\end{equation}
\end{small}
then
\begin{small}
\begin{equation}
\begin{aligned}\nonumber
\left\|\int^{\cdot}_{0}S_{\cdot u}g(y_{u})d\mathbf{w}_{u}\right\|_{\eta,0}
&\leq T^{\gamma-\eta}(|w|_{\gamma}+|w^{2}|_{2\gamma})\|g(y),(g(y))^{'}\|_{w,2\gamma,-2\gamma}\\
&\quad+T^{\gamma-\eta}\|g(y)\|_{\infty,0}|w|_{\gamma}+T^{2\gamma-\eta}\|(g(y))^{'}\|_{\infty,0}|w^{2}|_{2\gamma}.
\end{aligned}
\end{equation}
\end{small}
From \eqref{2.5} one has
$$\|g(y)\|_{\gamma,-2\gamma}\leq(1+|w|_{\gamma})(\|(g(y_{0}))^{'}\|_{\mathcal{H}^{d\times d}_{-2\gamma}}+T^{\gamma}\|g(y),(g(y))^{'}\|_{w,2\gamma,-2\gamma}).$$
Consequently, from above estimates we have
\begin{small}
\begin{equation}\label{2.10}
\begin{aligned}
&\left\|\int^{\cdot}_{0}S_{\cdot u}g(y_{u})d\mathbf{w}_{u},g(y)\right\|_{\mathcal{D}^{2\gamma,\eta}_{w}}\\
&\lesssim(1+|w|_{\gamma}+|w^{2}|_{2\gamma})\|(g(y_{0}))^{'}\|_{\mathcal{H}^{d\times d}_{-2\gamma}}+T^{\gamma}|w^{2}|_{2\gamma}\|(g(y))^{'}\|_{\gamma,-2\gamma}\\
\end{aligned}
\end{equation}
\end{small}
\begin{small}
\begin{equation}\nonumber
\begin{aligned}
&\quad\quad+\|g(y)\|_{\infty,0}+T^{\gamma-\eta}|w|_{\gamma}\|g(y)\|_{\infty,0}+T^{2\gamma-\eta}|w^{2}|_{2\gamma}\|(g(y))^{'}\|_{\infty,0}\\
&\quad\quad+\|g(y_{0})\|_{\mathcal{H}^{d}_{-2\gamma}}+(1+|w|_{\gamma}+|w^{2}|_{2\gamma})T^{\gamma}\|g(y),(g(y))^{'}\|_{w,2\gamma,-2\gamma}\\
&\quad\quad+(|w|_{\gamma}+|w^{2}|_{2\gamma})T^{\gamma-\eta}\|g(y),(g(y))^{'}\|_{w,2\gamma,-2\gamma}.
\end{aligned}
\end{equation}
\end{small}
Finally, using \eqref{2.10}, we easily obtain the desired result.
\end{proof}
\begin{lemma}\label{Lemma2.3}
Let $T>0$, $g\in\mathcal{C}^{3}_{-2\gamma,0}(\mathcal{H}, \mathcal{H}^{d})$, $(y,y^{'})$ and $(v,v^{'})\in\mathcal{D}^{2\gamma,\eta}_{w}([0,T];\mathcal{H})$, for some $\mathbf{w}\in\mathscr{C}^{\gamma}([0,T];\mathbb{R}^{d})$ and there exists $M>0$ such that $|w|_{\gamma}$, $|w^{2}|_{2\gamma}$, $\|y,y^{'}\|_{\mathcal{D}^{2\gamma,\eta}_{w}}$ and $\|v,v^{'}\|_{\mathcal{D}^{2\gamma,\eta}_{w}}\leq M$, then the following estimate holds true
\begin{small}
\begin{equation}\label{2.11}
\begin{aligned}
\|g(y)-g(v),(g(y)-g(v))^{'}\|_{\mathscr{D}^{2\gamma,2\gamma,0}_{S,w}}\leq C_{M,g,T}(1+|w|_{\gamma})^{2}\|y-v,(y-v)^{'}\|_{\mathcal{D}^{2\gamma,\eta}_{w}}.
\end{aligned}
\end{equation}
\end{small}
The constant $C_{M,g,T}$ depends on $M$, $g$ and the bounds of its derivatives. At the same time, it depends on time $T$, but can be chosen uniformly over $T\in(0,1]$.
\end{lemma}
\begin{proof}
First of all, we give an inequality which will be used throughout the proof:
for $g\in \mathcal{C}^{3}_{-2\gamma,0}(\mathcal{H}, \mathcal{H}^{d})$, $x_{1}$, $x_{2}$,$x_{3}$, $x_{4}\in\mathcal{H}_{\theta},\theta\geq-2\gamma$, the following bound holds
\begin{small}
\begin{equation}\label{2.12}
\begin{aligned}
&\|g(x_{1})-g(x_{2})-g(x_{3})+g(x_{4})\|_{\mathcal{H}^{d}_{\theta}}\\
&\leq C_{g}\left(\|x_{1}-x_{2}-x_{3}+x_{4}\|_{\mathcal{H}_{\theta}}+(\|x_{1}-x_{3}\|_{\mathcal{H}_{\theta}}+\|x_{2}-x_{4}\|_{\mathcal{H}_{\theta}})\|x_{3}-x_{4}\|_{\mathcal{H}_{\theta}}\right).
\end{aligned}
\end{equation}
\end{small}
Due to
\begin{small}
\begin{equation}
\begin{aligned}
&\|g(y_{t})-g(v_{t})-S_{ts}(g(y_{s})-g(v_{s}))\|_{\mathcal{H}^{d}_{-2\gamma}}\nonumber\\
&\leq\|g(y_{t})-g(v_{t})-(g(y_{s})-g(v_{s}))\|_{\mathcal{H}^{d}_{-2\gamma}}
+\|(S_{ts}-I)(g(y_{s})-g(v_{s}))\|_{\mathcal{H}^{d}_{-2\gamma}}\\
&\leq\|g(y_{t})-g(v_{t})-(g(y_{s})-g(v_{s}))\|_{\mathcal{H}^{d}_{-2\gamma}}
+\|(g(y_{s})-g(v_{s}))\|_{\mathcal{H}^{d}}|t-s|^{2\gamma},
\end{aligned}
\end{equation}
\end{small}
so, we have
$$\|g(y)-g(v)\|_{\gamma,-2\gamma}\leq|g(y)-g(v)|_{\gamma,-2\gamma}+T^{\gamma}\|g(y)-g(v)\|_{\infty,0}.$$
Similarly, we have
$$|g(y)-g(v)|_{\gamma,-2\gamma}\leq\|g(y)-g(v)\|_{\gamma,-2\gamma}+T^{\gamma}\|g(y)-g(v)\|_{\infty,0}.$$
Using \eqref{2.5} and \eqref{2.12}, we derive that
\begin{small}
\begin{equation}
\begin{aligned}
&|g(y)-g(v)|_{\gamma,-2\gamma}\leq C_{g}\left(|y-v|_{\gamma,-2\gamma}+(|y|_{\gamma,-2\gamma}+|v|_{\gamma,-2\gamma})\|y-v\|_{\infty,-2\gamma}\right)\nonumber\\
&\leq C_{g}(\|y-v\|_{\gamma,-2\gamma}+T^{\gamma}\|y-v\|_{\infty,0})\nonumber\\
&\quad+C_{g}(\|y\|_{\gamma,-2\gamma}+T^{\gamma}\|y\|_{\infty,0}+\|v\|_{\gamma,-2\gamma}+T^{\gamma}\|v\|_{\infty,0})\|y-v\|_{\infty,-2\gamma}\nonumber\\
&\leq C_{g}(\|y-v\|_{\gamma,-2\gamma}+T^{\gamma}\|y-v\|_{\infty,0})+C_{g}(1+|w|_{\gamma})(\|y^{'}_{0}\|_{\mathcal{H}^{d}_{-2\gamma}}+T^{\gamma}\|y_{0}\|_{\mathcal{H}}\nonumber\\
&\quad+T^{\gamma+\eta}\|y\|_{\eta,0}+T^{\gamma}\|y,y^{'}\|_{w,2\gamma,-2\gamma}+\|v^{'}_{0}\|_{\mathcal{H}^{d}_{-2\gamma}}+T^{\gamma}\|v_{0}\|_{\mathcal{H}}+T^{\gamma+\eta}\|v\|_{\eta,0}\\
&\quad+T^{\gamma}\|v,v^{'}\|_{w,2\gamma,-2\gamma})\|y-v\|_{\infty,-2\gamma}\\
&\leq C_{g,T,M}(1+|w|_{\gamma})(\|y-v\|_{\gamma,-2\gamma}+\|y-v\|_{\infty,-2\gamma}+T^{\gamma}\|y-v\|_{\infty,0}),
\end{aligned}
\end{equation}
\end{small}
therefore
\begin{small}
$$\|g(y)-g(v)\|_{\gamma,-2\gamma}\leq C_{g,T,M}(1+|w|_{\gamma})(\|y-v\|_{\gamma,-2\gamma}+\|y-v\|_{\infty,-2\gamma}+T^{\gamma}\|y-v\|_{\infty,0}).$$
\end{small}
Similarly, we can obtain that
\begin{small}
\begin{equation}
\begin{aligned}
\|Dg(y)y^{'}-Dy(v)v^{'}\|_{\gamma,-2\gamma}&\leq|Dg(y)y^{'}-Dg(v)v^{'}|_{\gamma,-2\gamma}\nonumber\\
&\quad+\|Dg(y)y^{'}-Dg(v)v^{'}\|_{\infty,0}T^{\gamma},
\end{aligned}
\end{equation}
\end{small}
\begin{small}
\begin{equation}
\begin{aligned}
|Dg(y)y^{'}-Dg(v)v^{'}|_{\gamma,-2\gamma}&\leq|Dg(y)(y^{'}-v^{'})|_{\gamma,-2\gamma}\nonumber\\
&\quad+|(Dg(y)-Dg(v))v^{'}|_{\gamma,-2\gamma}\\
&:=\uppercase\expandafter{\romannumeral1}+\uppercase\expandafter{\romannumeral2}.
\end{aligned}
\end{equation}
\begin{equation}
\begin{aligned}
\uppercase\expandafter{\romannumeral1}&\leq\|Dg(y)\|_{\infty,\mathcal{L}(\mathcal{H}_{-2\gamma}\otimes \mathbb{R}^{d},\mathcal{H}_{-2\gamma})}|y^{'}-v^{'}|_{\gamma,-2\gamma}+|Dg(y)|_{\gamma,\mathcal{L}(\mathcal{H}_{-2\gamma}\otimes \mathbb{R}^{d},\mathcal{H}_{-2\gamma})}\|y^{'}-v^{'}\|_{\infty,-2\gamma}\nonumber\\
&\leq C_{g}(|y^{'}-v^{'}|_{\gamma,-2\gamma}+|y|_{\gamma,-2\gamma}\|y^{'}-v^{'}\|_{\infty,-2\gamma})\nonumber\\
&\leq C_{g,M}(1+|w|_{\gamma})(\|y^{'}-v^{'}\|_{\gamma,-2\gamma}+\|y^{'}-v^{'}\|_{\infty,-2\gamma}+T^{\gamma}\|y^{'}-v^{'}\|_{\infty,0}),
\end{aligned}
\end{equation}
\end{small}
meanwhile, we have
$$\uppercase\expandafter{\romannumeral2}\leq C_{g,M,T}(1+|w|_{\gamma})(\|y-v\|_{\gamma,-2\gamma}+\|y-v\|_{\infty,-2\gamma}+T^{\gamma}\|y-v\|_{\infty,0}),$$
\begin{small}
\begin{equation}
\begin{aligned}\nonumber
\|Dg(y)y^{'}-Dg(v)v^{'}\|_{\infty,0}
&\leq\|Dg(y)(y^{'}-v^{'})\|_{\infty,0}+\|(Dg(y)-Dg(v))v^{'}\|_{\infty,0}\\
&\leq C_{g}(\|y^{'}-v^{'}\|_{\infty,0}+\|y-v\|_{\infty,0}\|v^{'}\|_{\infty,0})\\
&\leq C_{g,M}(\|y^{'}-v^{'}\|_{\infty,0}+\|y-v\|_{\infty,0})
\end{aligned}
\end{equation}
\end{small}
and $$\|y-v\|_{\infty,0}\lesssim T^{\eta}\|y-v\|_{\eta,0}+\|y_{0}-v_{0}\|_{\mathcal{H}},$$
according to above estimates, we obtain
$$\|Dg(y)y^{'}-Dg(v)v^{'}\|_{\gamma,-2\gamma}\leq C_{g,M,T}(1+|w|_{\gamma})\|y-v,(y-v)^{'}\|_{\mathcal{D}^{2\gamma,\eta}_{w}}.$$
Since
\begin{small}
\begin{equation}\nonumber
\begin{aligned}
R^{g(y)}_{t,s}
&=g(y_{t})-g(y_{s})-Dg(y_{s})S_{ts}y^{'}_{s}\delta w_{t,s}+ Dg(y_{s})S_{ts}y^{'}_{s}\delta w_{t,s}-Dg(y_{s})y^{'}_{s}\delta w_{t,s}\nonumber\\
&\quad+Dg(y_{s})y^{'}_{s}\delta w_{t,s}-S_{ts}Dg(y_{s})y^{'}_{s}\delta w_{t,s}+g(y_{s})-S_{ts}g(y_{s})\\
&=g(y_{t})-g(y_{s})-Dg(y_{s})\hat{\delta}y_{t,s}+Dg(y_{s})R^{y}_{t,s}+Dg(y_{s})(S_{ts}-I)y^{'}_{s}\delta w_{t,s}\\
&\quad-(S_{ts}-I)Dg(y_{s})y^{'}_{s}\delta w_{t,s}-(S_{ts}-I)g(y_{s})\\
&=g(y_{t})-g(y_{s})-Dg(y_{s})\delta y_{t,s}+Dg(y_{s})R^{y}_{t,s}+Dg(y_{s})(S_{ts}-I)y^{'}_{s}\delta w_{t,s}\\
&\quad-(S_{ts}-I)Dg(y_{s})y^{'}_{s}\delta w_{t,s}-(S_{ts}-I)g(y_{s})+Dg(y_{s})(S_{ts}-I)y_{s},
\end{aligned}
\end{equation}
\end{small}
hence, we have
\begin{small}
\begin{equation}\label{2.13}
\begin{aligned}
&R^{g(y)}_{t,s}-R^{g(v)}_{t,s}\\
&=g(y_{t})-g(y_{s})-Dg(y_{s})\delta y_{t,s}-(g(v_{t})-g(v_{s})-Dg(v_{s})\delta v_{t,s})\\
&\quad+Dg(y_{s})R^{y}_{t,s}-Dg(v_{s})R^{v}_{t,s}\\
&\quad-(S_{ts}-I)(g(y_{s})-g(v_{s}))\\
&\quad+Dg(y_{s})(S_{ts}-I)y^{'}_{s}\delta w_{t,s}-Dg(v_{s})(S_{ts}-I)v^{'}_{s}\delta w_{t,s}\\
\end{aligned}
\end{equation}
\end{small}
\begin{small}
\begin{equation}\nonumber
\begin{aligned}
&\quad-(S_{ts}-I)\left(Dg(y_{s})y^{'}_{s}\delta w_{t,s}-Dg(v_{s})v^{'}_{s}\delta w_{t,s}\right)\\
&\quad+Dg(y_{s})(S_{ts}-I)y_{s}-Dg(v_{s})(S_{ts}-I)v_{s}\\
&=\romannumeral1+\romannumeral2+\romannumeral3
+\romannumeral4+\romannumeral5+\romannumeral6.
\end{aligned}
\end{equation}
\end{small}
For $\romannumeral1$, applying (44) of \cite{MR4284415}, we have
\begin{small}
\begin{equation}
\begin{aligned}
&\|\romannumeral1\|_{\mathcal{H}^{d}_{-2\gamma}}\nonumber\\
&=\|\int^{1}_{0}\int^{1}_{0}C_{g}[\tau r^{2}(y_{t}-v_{t})+(r-\tau r^{2})(y_{s}-v_{s})]d\tau dr(\delta y_{t,s})^{2}\|_{\mathcal{H}_{-2\gamma}}\nonumber\\
&\quad+\|\int^{1}_{0}\int^{1}_{0}C_{g}[\tau r^{2}v_{t}+(r-\tau r^{2})v_{s}]d\tau dr\left((\delta y_{t,s})^{2}-(\delta v_{t,s})^{2}\right)\|_{\mathcal{H}_{-2\gamma}}\nonumber\\
&\leq C_{g}\|y-v\|_{\infty,-2\gamma}\|(\hat{\delta }y_{t,s}+(S_{ts}-I)y_{s})^{2}\|_{\mathcal{H}_{-2\gamma}}\nonumber\\
&\quad+C_{g}\|v\|_{\infty,-2\gamma}\|\hat{\delta }y_{t,s}+(S_{ts}-I)y_{s}+\hat{\delta}v_{t,s}+(S_{ts}-I)v_{s}\|_{\mathcal{H}_{-2\gamma}}\nonumber\\
&\quad\cdot\|\hat{\delta}(y-v)_{t,s}+(S_{ts}-I)(y-v)_{s}\|_{\mathcal{H}_{-2\gamma}}\nonumber\\
&\leq C_{g}\|y-v\|_{\infty,-2\gamma}(\|y\|_{\gamma,-2\gamma}|t-s|^{\gamma}+\|y\|_{\infty,0}|t-s|^{2\gamma})^{2}\nonumber\\
&\quad+C_{g}\|v\|_{\infty,-2\gamma}(\|y\|_{\gamma,-2\gamma}|t-s|^{\gamma}+\|y\|_{\infty,0}|t-s|^{2\gamma}+\|v\|_{\gamma,-2\gamma}|t-s|^{\gamma}\\
&\quad+\|v\|_{\infty,0}|t-s|^{2\gamma})\cdot(\|y-v\|_{\gamma,-2\gamma}|t-s|^{\gamma}+\|y-v\|_{\infty,0}|t-s|^{2\gamma}).
\end{aligned}
\end{equation}
\end{small}
For $\romannumeral2$, we have
\begin{small}
\begin{equation}
\begin{aligned}
\|\romannumeral2\|_{\mathcal{H}^{d}_{-2\gamma}}
&\leq\|Dg(y_{s})R^{y}_{t,s}-Dg(y_{s})R^{v}_{t,s}+Dg(y_{s})R^{v}_{t,s}-Dg(v_{s})R^{v}_{t,s}\|_{\mathcal{H}^{d}_{-2\gamma}}\nonumber\\
&\leq\|Dg(y_{s})(R^{y}_{t,s}-R^{v}_{t,s})\|_{\mathcal{H}^{d}_{-2\gamma}}+\|(Dg(y_{s})-Dg(v_{s}))R^{v}_{t,s}\|_{\mathcal{H}^{d}_{-2\gamma}}\nonumber\\
&\leq C_{g}|R^{y}-R^{v}|_{2\gamma,-2\gamma}|t-s|^{2\gamma}+C_{g}\|y-v\|_{\infty,-2\gamma}|R^{v}|_{2\gamma,-2\gamma}|t-s|^{2\gamma}.
\end{aligned}
\end{equation}
\end{small}
For $\romannumeral3$, we easily have
$$\|\romannumeral3\|_{\mathcal{H}^{d}_{-2\gamma}}\leq C_{g}\|y-v\|_{\infty,0}|t-s|^{2\gamma}.$$
For $\romannumeral4$, we have
\begin{small}
\begin{equation}
\begin{aligned}
\|\romannumeral4\|_{\mathcal{H}^{d}_{-2\gamma}}
&\leq\|(Dg(y_{s})-Dg(v_{s}))(S_{ts}-I)y^{'}_{s}\delta w_{t,s}\|_{\mathcal{H}^{d}_{-2\gamma}}\nonumber\\
&\quad+\|Dg(v_{s})(S_{ts}-I)(y^{'}_{s}-v^{'}_{s})\delta w_{t,s}\|_{\mathcal{H}^{d}_{-2\gamma}}\nonumber\\
&\leq C_{g}\|y-v\|_{\infty,-2\gamma}\|y^{'}\|_{\infty,0}|w|_{\gamma}|t-s|^{2\gamma}+C_{g}\|y^{'}-v^{'}\|_{\infty,0}|w|_{\gamma}|t-s|^{2\gamma}.
\end{aligned}
\end{equation}
\end{small}
For $\romannumeral5$ and $\romannumeral6$, similar to $\romannumeral4$, we obtain
$$\|\romannumeral5\|_{\mathcal{H}^{d}_{-2\gamma}}\leq C_{g,M}|w|_{\gamma}(\|y-v\|_{\infty,0}+\|y^{'}-v^{'}\|_{\infty,0})|t-s|^{3\gamma},$$
$$\|\romannumeral6\|_{\mathcal{H}^{d}_{-2\gamma}}\leq C_{g,M}(\|y-v\|_{\infty,0}+\|y-v\|_{\infty,-2\gamma})|t-s|^{2\gamma}.$$
Consequently, we easily obtain
$$|R^{g(y)}-R^{g(v)}|_{2\gamma,-2\gamma}\leq C_{g,M,T}(1+|w|_{\gamma})^{2}\|y-v,(y-v)^{'}\|_{\mathcal{D}^{2\gamma,\eta}_{w}}.$$

Finally, according to previous estimates and the norm of $\mathscr{D}^{2\gamma,2\gamma,0}_{S,w}([0,T];\mathcal{H})$, our result can be easily derived.
\end{proof}

Substituting \eqref{2.11} into \eqref{2.9}, we easily obtain the following result.
\begin{lemma}\label{Lemma2.4}
Let $T>0$, $g\in\mathcal{ C}^{3}_{-2\gamma,0}(\mathcal{H}, \mathcal{H}^{d})$, $(y,y^{'})$ and $(v,v^{'})\in\mathcal{D}^{2\gamma,\eta}_{w}([0,T];\mathcal{H})$, for some $\mathbf{w}\in\mathscr{C}^{\gamma}([0,T];\mathbb{R}^{d})$ and there exists $M>0$ such that $|w|$, $|w^{2}|$, $\|y,y^{'}\|_{\mathcal{D}^{2\gamma,\eta}_{w}}$ and $\|v,v^{'}\|_{\mathcal{D}^{2\gamma,\eta}_{w}}\leq M$, then, there exists a constant $C$ such that
\begin{small}
\begin{equation}\label{2.14}
\begin{aligned}
&\left\|\int^{\cdot}_{0}S_{\cdot u}(g(y_{u})-g(v_{u}))d\mathbf{w}_{u},g(y)-g(v)\right\|_{\mathcal{D}^{2\gamma,\eta}_{w}}\\
&\leq C_{g,M,T}(1+|w|_{\gamma}+|w^{2}|_{2\gamma})(1+|w|_{\gamma})^{2}\|y-v,(y-v)^{'}\|_{\mathcal{D}^{2\gamma,\eta}_{w}}.
\end{aligned}
\end{equation}
\end{small}
The constant $C_{M,g,T}$ depends on $M$, $g$ and the bounds of its derivatives, at the same time, it depends on time $T$, but is consistent with time $T\in(0,1]$.
\end{lemma}
However, in $\mathcal{D}^{2\gamma,\eta}_{w}([0,T];\mathcal{H})$, we also need to estimate the terms containing the initial condition and the drift of rough evolution equation \eqref{2.2}. Hence, we will focus on this in the following Lemma \ref{Lemma2.5}.
\begin{lemma}\label{Lemma2.5}
Let $T>0$, $\xi\in\mathcal{H}$, $f\in \mathcal{C}_{-2\gamma,0}(\mathcal{H}, \mathcal{H})$ be global Lipschitz continuous, and $(y,y^{'})\in\mathcal{D}^{2\gamma,\eta}_{w}([0,T];\mathcal{H})$, we have that the mildly Gubinelli derivative
\begin{equation}\label{2.15}
\begin{aligned}
 \left(S_{\cdot}\xi+\int^{\cdot}_{0}S_{\cdot u}f(y_{u})du\right)^{'}=0
\end{aligned}
\end{equation}
also have the estimate
\begin{equation}\label{2.16}
\begin{aligned}
\left\|S_{\cdot}\xi+\int^{\cdot}_{0}S_{\cdot u}f(y_{u})du,0\right\|_{\mathcal{D}^{2\gamma,\eta}_{w}}\leq C_{\gamma,T}(\|\xi\|+\|f(y)\|_{\infty,-2\gamma}+\|f(y)\|_{\infty,0}).
\end{aligned}
\end{equation}
Moreover, for two mildly controlled rough paths $(y,y^{'})$ and $(v,v^{'})$ with $y_{0}=\xi$ and $v_{0}=\tilde{\xi}$, we have
\begin{equation}\label{2.17}
\begin{aligned}
&\left\|S_{\cdot}(\xi-\tilde{\xi})+\int^{\cdot}_{0}S_{\cdot u}(f(y_{u})-f(v_{u}))du,0\right\|_{\mathcal{D}^{2\gamma,\eta}_{w}}\\
&\leq C_{\gamma,T}(\|\xi-\tilde{\xi}\|+\|f(y)-f(v)\|_{\infty,-2\gamma}+\|f(y)-f(v)\|_{\infty,0}).
\end{aligned}
\end{equation}
\end{lemma}
\begin{proof}
 Let $0<T\leq1$. Since $$\|S_{t}\xi-S_{ts}S_{s}\xi\|_{\mathcal{H}_{-2\gamma}}=0,$$
 $$\|S_{t}\xi-S_{ts}S_{s}\xi\|_{\mathcal{H}}=0,$$
 $$ \|S_{0}\xi\|_{\mathcal{H}}\lesssim\|\xi\|_{\mathcal{H}},$$  hence we have
 $$(S_{\cdot}\xi)^{'}=0,$$ $$ |R^{S_{\cdot}\xi}|_{2\gamma,-2\gamma}=0,$$
\begin{equation}\label{2.18}
 \|S_{\cdot}\xi,0\|_{\mathcal{D}^{2\gamma,\eta}_{w}}\leq C\|\xi\|.
\end{equation}
Meanwhile, due to
\begin{equation}
\begin{aligned}\nonumber
&\left\|\int^{t}_{0}S_{tu}f(y_{u})du-S_{ts}\int^{s}_{0}S_{su}f(y_{u})du\right\|_{\mathcal{H}_{-2\gamma}}\\\nonumber
&=\left\|\int^{t}_{s}S_{tu}f(y_{u})du\right\|_{\mathcal{H}_{-2\gamma}}\\
&\leq\int^{t}_{s}\|f(y_{u})\|_{\mathcal{H}_{-2\gamma}}du\\
&=\|f(y)\|_{\infty,-2\gamma}(t-s),
\end{aligned}
\end{equation}

$$\left\|\int^{0}_{0}S_{0u}f(y_{u})du\right\|_{\mathcal{H}_{-2\gamma}}=0,$$

$$\left\|\int^{t}_{s}S_{tu}f(y_{u})du\right\|_{\mathcal{H}} \leq\int^{t}_{s}\|f(y_{u})\|_{\mathcal{H}}du\leq (t-s)\|f(y)\|_{\infty,0},$$
thus we have $$\left(\int^{t}_{0}S_{tu}f(y_{u})du\right)^{'}=0,$$ $$\left\|\int^{\cdot}_{0}S_{\cdot u}f(y_{u})du\right\|_{\eta,0}\leq\|f(y)\|_{\infty,0}|t-s|^{1-\eta},$$

$$|R^{\int^{\cdot}_{0}S_{\cdot u}f(y_{u})du}|_{2\gamma,-2\gamma}\leq\|f(y)\|_{\infty,-2\gamma}(t-s)^{1-2\gamma},$$
\begin{equation}\label{2.19}
\left\|\int^{\cdot}_{0}S_{\cdot u}f(y_{u})du,0\right\|_{\mathcal{D}^{2\gamma,\eta}_{w}}\leq C_{\gamma}(T^{1-2\gamma}\|f(y)\|_{\infty,_{-2\gamma}}+T^{1-\eta}\|f(y)\|_{\infty,0}).
\end{equation}
 Finally, \eqref{2.16} is proved, consequently, \eqref{2.17} can be easily obtained.
\end{proof}

In $\mathcal{D}^{2\gamma,\eta}_{w}([0,T];\mathcal{H})$, base on above preliminary results, Similar to Theorem 4.1 of \cite{MR4040992} one can then easily derive a local solution for \eqref{2.2} by a fixed-point argument, i.e.:
\begin{theorem}\label{Theorem2.2}
 Let $T>0$, given $\xi\in\mathcal{H}$ and $\mathbf{w}=(w,w^{2})\in\mathscr{C}^{\gamma}([0,T];\mathbb{R}^{d})$ with $\gamma\in(\frac{1}{3},\frac{1}{2}]$. Then there exists $0<T_{0}\leq T$ such that the rough evolution equation \eqref{2.2} has a unique local solution represented by a mildly controlled rough path $(y,y^{'})\in\mathcal{D}^{2\gamma,\eta}_{w}([0,T_{0}];\mathcal{H})$ with $y^{'}=g(y)$, for all $0\leq t\leq T_{0}$
\begin{equation}\label{2.20}
\begin{aligned}
y_{t}=S_{t}\xi+\int^{t}_{0}S_{tu}f(y_{u})du+\int^{t}_{0}S_{tu}g(y_{u})d\mathbf{w}_{u}.
\end{aligned}
\end{equation}
\end{theorem}
\begin{proof}
Let $0<T\leq1$,
\begin{small}
$$\mathcal{M}(y,y^{'})_{t}=\left(S_{t}\xi+\int^{t}_{0}S_{tu}f(y_{u})du+\int^{t}_{0}S_{tu}g(y_{u})d\mathbf{w}_{u},g(y_{t})\right).$$
\end{small}
It is easy to obtain that if $(y_{0},y^{'}_{0})=(\xi,g(\xi))$ then the same is true for $\mathcal{M}(y,y^{'})$. Thus we can regard $\mathcal{M}_{T}$ as a mapping on the complete metric space:
\begin{small}
$$\{(y,y^{'})\in \mathcal{D}^{2\gamma,\eta}_{w}([0,T];\mathcal{H}):y_{0}=\xi,y^{'}_{0}=g(\xi)\}.$$
\end{small}
Meanwhile, since $$\|S_{\cdot}\xi+S_{\cdot}g(\xi)\delta w_{\cdot,0},S_{\cdot}g(\xi)\|_{w,2\gamma,-2\gamma}=0$$ hence we easily have that this is also true for the closed ball $B_{T}(w,r)$ centred at
\begin{small}
$t\rightarrow(S_{t}\xi+S_{t}g(\xi)\delta w_{t,0},S_{t}g(\xi))\in\mathcal{D}^{2\gamma,\eta}_{w}([0,T];\mathcal{H})$, i.e.
\end{small}
\begin{small}
\begin{equation}
\begin{aligned}\nonumber
B_{T}(w,r)&=\{(y,y^{'})\in \mathcal{D}^{2\gamma,\eta}_{w}([0,T];\mathcal{H}):y_{0}=\xi,y^{'}_{0}=g(\xi),  \|y-(S_{\cdot}\xi+S_{\cdot}g(\xi)\delta w_{\cdot,0})\|_{\eta,0}\\
          &\quad +\|y^{'}-S_{\cdot}g(\xi)\|_{\infty,0}+\|y-(S_{\cdot}\xi+S_{\cdot}g(\xi)\delta w_{\cdot,0}),y^{'}-S_{\cdot}g(\xi)\|_{w,2\gamma,-2\gamma}\leq r\}.
\end{aligned}
\end{equation}
\end{small}
Since, by triangle inequality, for $(y,y^{'})\in B_{T}(w,r)$ we have
 $$\|S_{\cdot}\xi+S_{\cdot}g(\xi)\delta w_{\cdot,0},S_{\cdot}g(\xi)\|_{w,2\gamma,2\gamma,\eta}\leq\|g(\xi)\|_{\mathcal{H}^{d}}+T^{\gamma-\eta}\|g(\xi)\|_{\mathcal{H}^{d}}|w|_{\gamma},$$ $$
\|y,y^{'}\|_{w,2\gamma,2\gamma,\eta}\leq r+\|g(\xi)\|_{\mathcal{H}^{d}}+\|g(\xi)\|_{\mathcal{H}^{d}}|w|_{\gamma}.
$$
 Then, one obtains
\begin{small}
\begin{equation}
\begin{aligned}\nonumber
&\|\mathcal{M}(y)-(S_{\cdot}\xi+S_{\cdot}g(\xi)\delta w_{\cdot,0}),g(y)-S_{\cdot}g(\xi)\|_{w,2\gamma,2\gamma,\eta}\\
&\leq\|\mathcal{M}(y),g(y)\|_{w,2\gamma,2\gamma,\eta}+\|S_{\cdot}\xi+S_{\cdot}g(\xi)\delta w_{\cdot,0},S_{\cdot}g(\xi)\|_{W,2\gamma,2\gamma,\eta}\\
&\leq
T^{1-2\gamma}(\|f(y)\|_{\infty,_{-2\gamma}}+\|f(y)\|_{\infty,0})+T^{\gamma-\eta}\|g(y)\|_{\infty,0}|w|_{\gamma}+\|g(y)\|_{\infty,0}\\
&\quad+T^{\gamma-\eta}(|w|_{\gamma}+|w^{2}|_{2\gamma})\|g(y),(g(y))^{'}\|_{w,2\gamma,-2\gamma}+(1+|w|_{\gamma}+|w^{2}|_{2\gamma})\|(g(y_{0}))^{'}\|_{\mathcal{H}^{d\times d}_{-2\gamma}}\\
&\quad+T^{2\gamma-\eta}\|(g(y))^{'}\|_{\infty,0}|w^{2}|_{2\gamma}+T^{\gamma}(1+|w|_{\gamma}+|w^{2}|_{2\gamma})\|g(y),(g(y))^{'}\|_{w,2\gamma,-2\gamma}\\
&\quad+T^{\gamma}|w^{2}|_{2\gamma}\|(g(y))^{'}\|_{\gamma,-2\gamma}+\|g(\xi)\|_{\mathcal{H}^{d}}+T^{\gamma-\eta}\|g(\xi)\|_{\mathcal{H}^{d}}|w|_{\gamma},
\end{aligned}
\end{equation}
\end{small}
since $g\in\mathcal{C}^{3}_{-2\gamma,0}(\mathcal{H}, \mathcal{H}^{d})$ and $\|g(0)\|_{\mathcal{H}^{d}_{\theta}}$ for $\theta\geq-2\gamma$, by mean value theorem we easily obtain $\|g(y)\|_{\mathcal{H}^{d}_{\theta}}\leq 1+\|y\|_{\mathcal{H}_{\theta}}$, consequently, base on \eqref{2.8} and above estimates we easily have
\begin{small}
\begin{equation}
\begin{aligned}\nonumber
&\|\mathcal{M}(y)-(S_{\cdot}\xi+S_{\cdot}g(\xi)\delta w_{\cdot,0}),g(y)-S_{\cdot}g(\xi)\|_{w,2\gamma,2\gamma,\eta}\\
&\leq C_{L_{f},g,\varrho(w,0)}+T^{\gamma-\eta}C_{g,\varrho(w,0),\|\xi\|,\|y,y^{'}\|_{w,2\gamma,2\gamma,\eta}}.
\end{aligned}
\end{equation}
\end{small}
Let $r=2C_{L_{f},g,\varrho(w,0)}$, then for $\forall (y,y^{'})\in B_{T}(w,r)$, we have
\begin{equation}
\begin{aligned}\nonumber
\|\mathcal{M}(y)-(S_{\cdot}\xi+S_{\cdot}g(\xi)\delta w_{\cdot,0}),g(y)-S_{\cdot}g(\xi)\|_{w,2\gamma,2\gamma,\eta}
\leq\frac{r}{2}+T^{\gamma-\eta}C_{g,\varrho(w,0),\|\xi\|,r}.
\end{aligned}
\end{equation}
By letting $T=T_{0}$ to be sufficient small such that
$$T_{0}^{\gamma-\eta}C_{g,\varrho(w,0),\|\xi\|,r}<\frac{r}{2},$$
then one otains
$$\mathcal{M}(B_{T_{0}}(w,r))\subseteq B_{T_{0}}(w,r).$$
For the sake of proving contractivity of $\mathcal{M}_{T_{0}}$, one can use steps that are similar to the previous steps to show
\begin{small}
\begin{equation}
\begin{aligned}\nonumber
\|\mathcal{M}(y)-\mathcal{M}(v),g(y)-g(v)\|_{w,2\gamma,2\gamma,\eta}
\leq T_{0}^{\eta\wedge(\gamma-\eta)}C_{g,\rho(w,0),\|\xi\|,r}\|y-v,y^{'}-v^{'}\|_{w,2\gamma,2\gamma,\eta}.
\end{aligned}
\end{equation}
\end{small}
This ensures contractivity when $T_{0}$ is sufficient small. By the Banach fixed point theorem,  one has that there is a unique $(y,y^{'})\in B_{T_{0}}(w,r)$ satisfies $\mathcal{M}(y)=y$, i.e. a solution to REE \eqref{2.2} on the small time interval $[0,T_{0}]$.
\end{proof}
\begin{remark}\label{remark5}
Our proof of Theorem \ref{Theorem2.2} is simpler than the one of Theorem 4.1 in \cite{MR4040992}. This is also the key why we choose to study \eqref{2.2} in the space $\mathcal{D}^{2\gamma,\eta}_{w}([0,T];\mathcal{H})$. Here, we directly view $\mathcal{M}_{T}$ as mapping from the space $\mathcal{D}^{2\gamma,\eta}_{w}([0,T];\mathcal{H})$ into itself, however, in \cite{MR4040992},  the technique is to take $\varepsilon\in(1/3,\gamma]$ and view $\mathcal{M}_{T}$ as mapping from $\mathscr{D}^{2\gamma,2\gamma,\varepsilon}_{S,w}([0,T];\mathcal{H}_{2\varepsilon-2\gamma})$, rather than $\mathscr{D}^{2\gamma,2\gamma,\gamma}_{S,w}([0,T];\mathcal{H})$.
\end{remark}
\subsection{Global in time solution of rough evolution equation}
As we all known, the global in time solution is the key that allows one to consider the longtime behaviour of rough evolution equation \eqref{2.2}, so in this subsection we will focus on this issue. Similar to \cite{MR4097587} and \cite{MR11111}, we will derive the following result which is fundamental importance for the discussion of global in time solution for \eqref{2.2}. according to \eqref{2.8}, \eqref{2.10}, \eqref{2.18} and \eqref{2.19}, we obtain that:
\begin{corollary}\label{Corollary2.1}
Let $(y,g(y))\in\mathcal{D}^{2\gamma,\eta}_{w}([0,T];\mathcal{H})$ with $0<T\leq1$ be the solution of \eqref{2.2} with the initial condition $y_{0}=\xi\in\mathcal{H}$. Then one has the following estimate
\begin{equation}\label{2.21}
\|y,g(y)\|_{\mathcal{D}^{2\gamma,\eta}_{w}}\lesssim1+\|\xi\|+T^{\gamma-\eta}\|y,g(y)\|_{\mathcal{D}^{2\gamma,\eta}_{w}}.
\end{equation}
\end{corollary}

\begin{proof}
Since $(y,g(y))$ is the solution of \eqref{2.2} we have
\begin{small}
\begin{equation}
\begin{aligned}\nonumber
\|y,g(y)\|_{\mathcal{D}^{2\gamma,\eta}_{w}}&\leq \|S_{\cdot}\xi,0\|_{\mathcal{D}^{2\gamma,\eta}_{w}}+\left\|\int^{\cdot}_{0}S_{\cdot u}f(y_{u})du,0\right\|_{\mathcal{D}^{2\gamma,\eta}_{w}}\\
&\quad+\left\|\int^{\cdot}_{0}S_{\cdot u}g(y_{u})d\mathbf{w}_{u},g(y)\right\|_{\mathcal{D}^{2\gamma,\eta}_{w}}.
\end{aligned}
\end{equation}
\end{small}
According to \eqref{2.10}, \eqref{2.18} and \eqref{2.19}, we obtain
\begin{small}
\begin{equation}
\begin{aligned}\nonumber
&\|y,g(y)\|_{\mathcal{D}^{2\gamma,\eta}_{w}}\\ &\lesssim\|\xi\|+T^{1-2\gamma}(\|f(y)\|_{\infty,-2\gamma}+\|f(y)\|_{\infty,0})+\|g(y_{0})\|_{\mathcal{H}^{d}_ {-2\gamma}}+\|(g(y_{0}))^{'}\|_{\mathcal{H}^{d\times d}_{-2\gamma}}\\
&\quad+\|g(y)\|_{\infty,0}+T^{2\gamma-\eta}\|(g(y))^{'}\|_{\infty,0}+T^{\gamma-\eta}\|g(y),(g(y))^{'}\|_{w,2\gamma,-2\gamma}.
\end{aligned}
\end{equation}
\end{small}
Meanwhile, from \eqref{2.8} we have
$$\|g(y),(g(y))^{'}\|_{w,2\gamma,-2\gamma}\lesssim1+\|y,g(y)\|_{\mathcal{D}^{2\gamma,\eta}_{w}}.$$
Combining $\|g(y)\|_{\mathcal{H}^{d}_{\theta}}\leq 1+\|y\|_{\mathcal{H}_{\theta}}$ and the bounds of its derivatives with previous estimates, we obtain
\begin{equation}
\begin{aligned}\nonumber
\|y,g(y)\|_{\mathcal{D}^{2\gamma,\eta}_{w}}&\lesssim 1+\|\xi\|+T^{1-2\gamma}(1+\|y\|_{\infty,-2\gamma}+\|y\|_{\infty,0})\\
&\quad+T^{2\gamma-\eta}\|(g(y))^{'}\|_{\infty,0}+T^{\gamma-\eta}\|g(y),(g(y))^{'}\|_{w,2\gamma,-2\gamma}\\
&\lesssim1+\|\xi\|+T^{2\gamma-\eta}(1+\|y\|_{\infty,-2\gamma}+\|y\|_{\infty,0}+\|g(y)\|_{\infty,0})\\
&\quad+T^{\gamma-\eta}\|y,g(y)\|_{\mathcal{D}^{2\gamma,\eta}_{w}}\\
&\lesssim1+\|\xi\|+T^{\gamma-\eta}\|y,g(y)\|_{\mathcal{D}^{2\gamma,\eta}_{w}}.
\end{aligned}
\end{equation}
Finally we obtain the desired result.
\end{proof}

Applying a concatenation discussion of \cite{MR4097587} and \cite{MR11111}, according to \eqref{2.21} we obtain an a-priori bound for the solution of \eqref{2.2}. The technique of proof is identical to the one of \cite{MR4097587} Lemma 5.8, we omit here.
\begin{lemma}\label{Lemma2.6}
Let $T>0$, $(y,g(y))\in\mathcal{D}^{2\gamma,\eta}_{w}([0,T];\mathcal{H})$ be the solution of \eqref{2.2}, where the initial condition $y_{0}=\xi\in\mathcal{H}$ with $\|\xi\|\leq\rho$. Let $\tilde{r}=1\vee\rho$, then there exists constant $M$ such that
$$\|y\|_{\infty,0,[0,T]}\leq M\tilde{r}e^{MT}.$$
\end{lemma}
%\begin{proof}
%For any $\bar{T}\in(0,T]$, the restriction of $(Y,G(Y))$ on $[0,\bar{T}]$ is a solution of (2.2) on $[0,\bar{T}]$. Then by (2.22), for $0<\bar{T}\leq1$, we have that there exists a constant $C\geq1$ such that
  % $$\|Y,G(Y)\|_{\mathcal{D}^{2\gamma,\eta}_{W,[0,\bar{T}]}}\leq C(2\tilde{r}+\bar{T}^{\gamma-\eta}\|Y,G(Y)\|_{\mathcal{D}^{2\gamma,\eta}_{W,[0,\bar{T}]}}).$$
%So we can choose $0<\tilde{T}\leq\bar{T}$ small enough satisfying $C\tilde{T}^{\gamma-\eta}\leq\frac{1}{2}$ one has the bound
%$$\|Y\|_{\infty,0,[0,\tilde{T}]}\leq\|Y,G(Y)\|_{\mathcal{D}^{2\gamma,\eta}_{W,[0,\tilde{T}]}}\leq 4C\tilde{r}.$$
%It is important to note that the choice of $\tilde{T}$ is independent of $\tilde{r}$ and $T$.

%If $CT^{\gamma-\eta}\leq\frac{1}{2}$ the previous estimate holds true by choosing $M\geq4C$.
%Otherwise, we choose $N\in\mathbb{N}$ such that $\frac{1}{4}\leq C(\frac{T}{N})^{\gamma-\eta}\leq\frac{1}{2}$, this is well-defined since $\gamma-\eta<1$. Then
% $$\|Y\|_{\infty,0,[0,\frac{T}{N}]}\leq 4C\tilde{r}.$$
%Meanwhile by a concatenation argument one has that for all $k\in\{0,\cdot\cdot\cdot,N-1\}$
 %$$\|Y\|_{\infty,0,[\frac{k}{N}T,\frac{k+1}{N}T]}\leq (4C)^{k+1}\tilde{r}.$$
 %Therefore,
% $$\|Y\|_{\infty,0,[0,T]}\leq\max_{k\in\{0,\cdot\cdot\cdot,N-1\}}\|Y\|_{\infty,0,[\frac{k}{N}T,\frac{k+1}{N}T]}\leq (4C)^{N}\tilde{r}.$$
% At last, according to $\frac{1}{4}\leq C(\frac{T}{N})^{\gamma-\eta}$, we obtain $N\leq(4C)^{\frac{1}{\gamma-\eta}}T$, hence we obtain the desired result with a sufficiently large $M$.
%\end{proof}
Lemma \ref{Lemma2.6} guarantees that the solution of \eqref{2.2} does not explode in any finite time, therefore, in $\mathcal{D}^{2\gamma,\eta}_{w}([0,T];\mathcal{H})$, according to above preliminary results, base on Theorem \ref{Theorem2.2}, we have the following result that the local solution of \eqref{2.2} can be extended to global one by a standard concatenation discussion, the details of proof one can refer to \cite{MR4097587} Theorem 5.10 and \cite{MR11111} Theorem 3.9, we omit here.
\begin{theorem}\label{theorem3}
Let $T>0$, given $\xi\in\mathcal{H}$ and $\mathbf{w}=(w,w^{2})\in\mathscr{C}^{\gamma}([0,T];\mathbb{R}^{d})$. The rough evolution equation \eqref{2.2} has a unique global solution represented by a mildly controlled rough path $(y,y^{'})\in\mathcal{D}^{2\gamma,\eta}_{w}([0,T];\mathcal{H})$ given by
\begin{equation}\label{2.22}
\begin{aligned}
(y,y^{'})=\left(S_{\cdot}\xi+\int^{\cdot}_{0}S_{\cdot u}f(y_{u})du+\int^{\cdot}_{0}S_{\cdot u}g(y_{u})d\mathbf{w}_{u},g(y)\right).
\end{aligned}
\end{equation}
\end{theorem}
%\begin{proof}
%For the initial condition $Y_{0}=\xi\in\mathcal{H}$ with $\|\xi\|\leq\rho$, let $\tilde{r}=1\vee\rho$. By Lemma 2.6  each solution of (2.2) can be bounded by
%$$\|Y\|_{\infty,0,[0,T]}\leq M\tilde{r}e^{MT}=:\hat{r}.$$
%Particular, this means that $\|Y_{t}\|_{\mathcal{H}}\leq\hat{r}$ for all $t\leq T$. Applying Theorem 2.2 with $\|\xi\|\leq\hat{r}$ entails the existence of a unique local solution on time $[0,T_{0}]$, where $T_{0}(\hat{r})$. For simplicity, since we can choose $T_{0}$ arbitrary small, we set $N:=\frac{T}{T_{0}}\in\mathbb{N}$ for $N\geq2$.

%Since
%$\|Y_{\frac{T}{N}}\|\leq\hat{r}$, we furthermore obtain the existence of a unique local solution of (2.2) on the time
%interval$[0,\frac{T}{N}]$ with initial data $Y_{\frac{T}{N}}$. Concatenating both solutions provides a solution of (2.2) on the
%time interval $[0,2\frac{T}{N}]$ with initial data $Y_{0}$. Iterating this argument one can construct a solution on
%the whole time interval $[0,T]$.
%\end{proof}
\begin{remark}\label{remark4}
We emphasis the fact that the solution of \eqref{2.2} is global in time. Rough paths and rough drivers are usually defined on compact intervals, according to \cite{MR3624539} and \cite{MR11111}, we say $\mathbf{w}=(w,w^{2})\in\mathscr{C}^{\gamma}(\mathbb{R};\mathbb{R}^{d})$ is a $\gamma$-H\"{o}lder rough path if $\mathbf{w}|_{I}\in\mathscr{C}^{\gamma}(I;\mathbb{R}^{d})$ for every compact interval $I\subseteq\mathbb{R}$ containing 0. Hence, in our setting, we have that $(y,y^{'})\in\mathcal{D}^{2\gamma,\eta}_{w}([0,\infty);\mathcal{H})$ if $(y,y^{'})\in\mathcal{D}^{2\gamma,\eta}_{w}([0,T];\mathcal{H})$ for every $T>0$, therefore, we set $C_{g}|_{[0,\infty)}=\mathop{\max}\limits_{I\subseteq[0,\infty)}C_{g}|_{I},\varrho(w,0)|_{[0,\infty)}=\mathop{\max}\limits_{I\subseteq[0,\infty)}\varrho(w,0)|_{I}$, according to Theorem 2.2, let  $r=2C_{L_{f},g,\varrho(w,0)}|_{[0,\infty)}$ and previous deliberations we have that $r$ is keep invariant in concatenation arguments and one can obtain a unique solution of \eqref{2.2} in $\mathcal{D}^{2\gamma,\eta}_{w}([0,\infty);\mathcal{H})$.
\end{remark}
%In order to derive uniform explicit estimates of the constants depending on the $\gamma$, the coefficients $F$ and $G$ and the random input $\mathbf{W}$. These will be required in the Section 4. From now on, in the following we are working on time interval $[0,1]$.
\subsection{Truncated rough evolution equation}
We will prove a  global unstable manifold for a version of \eqref{2.2} with cut-off over a random neighborhood in the following section 4. Hence we will construct a local unstable manifolds depend on the size of perturbations and the spectral gap of the linear part of \eqref{2.2}. In order to consider the existence of local invariant manifold by using Lyapunov-Perron method, in this subsection, we modify these nonlinear $f$ and $g$ by applying appropriate cut-off techniques to make their Lipschitz constants small enough. Since in contrast to the classical cut-off techniques (as in \cite{MR3376184}, \cite{MR2016614}, \cite{MR2110052} and so on), in our case, similar to \cite{MR11112} and \cite{MR4284415}, we truncate the norm of mildly controlled rough path $(y,y^{'})$. Due to the technical reasons of Lyapunov-Perron method, which we will use in section 4, we fix the time interval as $[0,1]$ in this subsection.

Meanwhile, we assume the following restrictions on the drift and diffusion coefficients:
\begin{itemize}
\item $f\in\mathcal{ C}^{1}_{-2\gamma,0}(\mathcal{H}, \mathcal{H})$ is global Lipschitz continuous with $ f(0)=Df(0)=0;$

\item $g\in\mathcal{ C}^{3}_{-2\gamma,0}(\mathcal{H}, \mathcal{H}^{d})$  with $ g(0)=Dg(0)=D^{2}g(0)=0,$
\end{itemize}
so one easily obtains that $(y=0,y^{'}=0)$ is a stationary solution of \eqref{2.2}.

Let $\chi:\mathcal{D}^{2\gamma,\eta}_{w}([0,1];\mathcal{H})\rightarrow\mathcal{D}^{2\gamma,\eta}_{w}([0,1];\mathcal{H})$ be a  Lipschitz continuous cut-off function:
$$\chi(y):=\left\{
\begin{aligned}
y,& \qquad \|y,y^{'}\|_{\mathcal{D}^{2\gamma,\eta}_{w}}\leq\frac{1}{2},    \\
0,& \qquad \|y,y^{'}\|_{\mathcal{D}^{2\gamma,\eta}_{w}}\geq1.
\end{aligned}
\right.
$$
As examples in subsection 2.1 of \cite{MR4284415}, we can take $\varphi:\mathbb{R}^{+}\rightarrow[0,1]$ is a $C^{3}_{b}$ Lipschitz cut-off function,
%$$f(x):=\left\{
%\begin{aligned}
%1,& \qquad x\leq\frac{1}{2},    \\
%2-2x,& \qquad x\in(\frac{1}{2},1),\\
%0,&\qquad x\geq1,
%\end{aligned}
%\right.
%$$
then $\chi(y)$ can be constructed as
$$\chi(y)=y\varphi(\|y,y^{'}\|_{\mathcal{D}^{2\gamma,\eta}_{w}}).$$
In the following, we assume that $\chi$ is constructed by $\varphi$. According to Definition \ref{Definition2.2}, one has
$$\chi^{'}(y)=y^{'}\varphi(\|y,y^{'}\|_{\mathcal{D}^{2\gamma,\eta}_{w}}),$$
this construction indicates that
$$(\chi(y),\chi^{'}(y)):=\left\{
\begin{aligned}
(y,y^{'}),& \qquad \|y,y^{'}\|_{\mathcal{D}^{2\gamma,\eta}_{w}}\leq\frac{1}{2},    \\
0,& \qquad \|y,y^{'}\|_{\mathcal{D}^{2\gamma,\eta}_{w}}\geq1.
\end{aligned}
\right.
$$
For a positive number $R$, we define
$$\chi_{R}(y)=R\chi(y/R),$$
this means that
$$\chi_{R}(y):=\left\{
\begin{aligned}
y,& \qquad \|y,y^{'}\|_{\mathcal{D}^{2\gamma,\eta}_{w}}\leq\frac{R}{2},    \\
0,& \qquad \|y,y^{'}\|_{\mathcal{D}^{2\gamma,\eta}_{w}}\geq R,
\end{aligned}
\right.
$$
then
$$(\chi_{R}(y),\chi^{'}_{R}(y)):=\left\{
\begin{aligned}
(y,y^{'}),& \qquad \|y,y^{'}\|_{\mathcal{D}^{2\gamma,\eta}_{w}}\leq\frac{R}{2},    \\
0,& \qquad \|y,y^{'}\|_{\mathcal{D}^{2\gamma,\eta}_{w}}\geq R.
\end{aligned}
\right.
$$
For a mildly controlled rough path $(y,y^{'})\in\mathcal{D}^{2\gamma,\eta}_{w}([0,1];\mathcal{H})$, we introduce the operators
$$f_{R}(y_{t}):=f\circ \chi_{R}(y_{t}),\qquad g_{R}(y_{t}):=g\circ \chi_{R}(y_{t}).$$
Base on Lemma \ref{Lemma2.1}, we obtain the mildly Gubinelli derivative of $g_{R}(y)$ :
$$(g_{R}(y))^{'}=Dg(\chi_{R}(y))\chi_{R}^{'}(y)=Dg(y \varphi(\|y,y^{'}\|_{\mathcal{D}^{2\gamma,\eta}_{w}}/R))y^{'} \varphi(\|y,y^{'}\|_{\mathcal{D}^{2\gamma,\eta}_{w}}/R).$$
It directly obtained that if $\|y,y^{'}\|_{\mathcal{D}^{2\gamma,\eta}_{w}}\leq R/2$, we have that $f_{R}(y)=f(y)$ and $g_{R}(y)=g(y)$.

%Before we proceed to next step as in \cite{MR4284415}, we give the following inequality which will be frequently used in following Lemmas. Let $H\in\mathcal{ C}^{3}_{-2\gamma,0}(\mathcal{H}, \mathcal{H}^{d})\left(\mathcal{ C}^{1}_{-2\gamma,0}(\mathcal{H}, \mathcal{H})\right)$, with $H(0)=DH(0)=0$, for $x,p\in\mathcal{H}_{\alpha}, \alpha\geq-2\gamma$, one has
%$$H(x)-H(p)=(x-p)\int^{1}_{0}DH(rx-(1-r)p)dr.$$
%Furthermore, since $DH(0)=0$, we also have

%\begin{equation}
%\begin{aligned}
%&\|H(x)-H(p)\|_{\mathcal{H}^{d}_{\alpha}/\mathcal{H}_{\alpha}}\\
%&=\|x-p\|_{\mathcal{H}_{\alpha}}\int^{1}_{0}\|DH(rx-(1-r)p)\|_{\mathcal{L}(\mathcal{H}_{\alpha}\otimes\mathbb{R}^{d},\mathcal{H}_{\alpha})/\mathcal{L}(\mathcal{H}_{\alpha},\mathcal{H}_{\alpha})}dr\\
%&\leq C_{H}\max\{\|x\|_{\mathcal{H}_{\alpha}},\|p\|_{\mathcal{H}_{\alpha}}\}\|x-p\|_{\mathcal{H}_{\alpha}}.
%\end{aligned}
%\end{equation}
%When $H\in\mathcal{ C}^{3}_{-2\gamma,0}(\mathcal{H}, \mathcal{H}^{d})$, from the assumption $D^{2}H(0)=0$, similar to (2.23), one can easily derived the difference $DH(x)-DH(p)$, i.e.
%$$DH(x)-DH(p)=(x-p)\int^{1}_{0}D^{2}H(rx-(1-r)p)dr.$$ and
%$$\|DH(x)-DH(p)\|_{\mathcal{L}(\mathcal{H}_{\alpha}\otimes\mathbb{R}^{d},\mathcal{H}_{\alpha})}\leq C_{H}\max\{\|x\|_{\mathcal{H}_{\alpha}},\|p\|_{\mathcal{H}_{\alpha}}\}\|x-p\|_{\mathcal{H}_{\alpha}}.$$

Next, we will discuss the Lipschitz continuity of $f_{R}$ and $g_{R}$, and the Lipschitz constants are supposed to be strictly increasing in $R$.
\begin{lemma}\label{Lemma2.7}
Let $(y,y^{'})$ and $(v,v^{'})\in\mathcal{D}^{2\gamma,\eta}_{w}([0,1];\mathcal{H})$, then there exists a constant $C=C_{f,\chi,|w|_{\gamma}}$ such that
\begin{equation}\label{2.23}
\begin{aligned}
\|f_{R}(y)-f_{R}(v)\|_{\infty,0}+\|f_{R}(y)-f_{R}(v)\|_{\infty,-2\gamma}\leq CR\|y-v,y^{'}-v^{'}\|_{\mathcal{D}^{2\gamma,\eta}_{w}}.
\end{aligned}
\end{equation}
\end{lemma}

\begin{proof}
 We easily have
 $$\sup_{t\in[0,1]}\|f_{R}(y_{t})-f_{R}(v_{t})\|_{\mathcal{H}_{-2\gamma}}=\sup_{t\in[0,1]}\|f(\chi_{R}(y_{t}))-f(\chi_{R}(v_{t}))\|_{\mathcal{H}_{-2\gamma}}.$$
Firstly, since $f\in\mathcal{ C}^{1}_{-2\gamma,0}(\mathcal{H}, \mathcal{H})$ is global Lipschitz continuous and $Df(0)=0$, thus we have
\begin{equation}\nonumber
\begin{aligned}
&\|f(\chi_{R}(y_{t}))-f(\chi_{R}(v_{t}))\|_{\mathcal{H}_{-2\gamma}}\\
&\leq\int^{1}_{0}\|Df(r\chi_{R}(y_{t})+(1-r)\chi_{R}(v_{t}))\|_{\mathcal{L}(\mathcal{H}_{-2\gamma},\mathcal{H}_{-2\gamma})}dr \|\chi_{R}(y_{t})-\chi_{R}(v_{t})\|_{\mathcal{H}_{-2\gamma}}\\
&\leq C_{f}\max\{\|\chi_{R}(y_{t})\|_{\mathcal{H}_{-2\gamma}},\|\chi_{R}(v_{t})\|_{\mathcal{H}_{-2\gamma}}\}\|\chi_{R}(y_{t})-\chi_{R}(v_{t})\|_{\mathcal{H}_{-2\gamma}}.
\end{aligned}
\end{equation}
Secondly, due to $$\|y\|_{\infty,-2\gamma}\leq\|y_{0}\|_{\mathcal{H}_{-2\gamma}}+\|y\|_{\gamma,-2\gamma}\leq\|y_{0}\|_{\mathcal{H}}+\|y\|_{\gamma,-2\gamma}$$
and
\begin{equation}\nonumber
\begin{aligned} \|y\|_{\gamma,-2\gamma}&\leq(1+|w|_{\gamma})(\|y^{'}_{0}\|_{\mathcal{H}^{d}_{-2\gamma}}+\|y,y^{'}\|_{w,2\gamma,-2\gamma})\\
&\leq(1+|w|_{\gamma})\|y,y^{'}\|_{\mathcal{D}^{2\gamma,\eta}_{w}},
\end{aligned}
\end{equation}
 hence we obtain
\begin{equation}\nonumber
\begin{aligned}
\|y\|_{\infty,-2\gamma}&\leq \|y_{0}\|_{\mathcal{H}}+(1+|w|_{\gamma})(\|y^{'}_{0}\|_{\mathcal{H}^{d}_{-2\gamma}}+\|y,y^{'}\|_{w,2\gamma,-2\gamma})\\
&\leq(1+|w|_{\gamma})(\|y_{0}\|_{\mathcal{H}}+\|y^{'}_{0}\|_{\mathcal{H}^{d}_{-2\gamma}}+\|y,y^{'}\|_{w,2\gamma,-2\gamma})\\
&\leq(1+|w|_{\gamma})\|y,y^{'}\|_{\mathcal{D}^{2\gamma,\eta}_{w}}.
\end{aligned}
\end{equation}
In addition,
\begin{equation}\label{2.24}
\begin{aligned}
\|\chi_{R}(y)\|_{\infty,-2\gamma}&=\|y\varphi(\|y,y^{'}\|_{\mathcal{D}^{2\gamma,\eta}_{w}}/R)\|_{\infty,-2\gamma}
=\|y\|_{\infty,-2\gamma}\varphi(\|y,y^{'}\|_{\mathcal{D}^{2\gamma,\eta}_{w}}/R)\\
&\leq C(1+|w|_{\gamma})\|y,y^{'}\|_{\mathcal{D}^{2\gamma,\eta}_{w}}\leq C_{|w|_{\gamma}}R,
\end{aligned}
\end{equation}
and $\varphi:\mathbb{R}^{+}\rightarrow[0,1]$ is $C^{3}_{b}$, then we have
\begin{equation}
\begin{aligned}\nonumber
\|\chi_{R}(y_{t})-\chi_{R}(v_{t})\|_{\mathcal{H}_{-2\gamma}}&=\|y_{t}\varphi(\|y,y^{'}\|_{\mathcal{D}^{2\gamma,\eta}_{w}}/R)-v_{t}\varphi(\|v,v^{'}\|_{\mathcal{D}^{2\gamma,\eta}_{w}}/R)\|_{\mathcal{H}_{-2\gamma}}\\
&\leq\|(y_{t}-v_{t})\varphi(\|y,y^{'}\|_{\mathcal{D}^{2\gamma,\eta}_{w}}/R)\|_{\mathcal{H}_{-2\gamma}}\\
&\quad+\|v_{t}(\varphi(\|y,y^{'}\|_{\mathcal{D}^{2\gamma,\eta}_{w}}/R)-\varphi(\|v,v^{'}\|_{\mathcal{D}^{2\gamma,\eta}_{w}}/R))\|_{\mathcal{H}_{-2\gamma}}\\
&\leq \|y-v\|_{\infty,-2\gamma}\\
&\quad+\|v\|_{\infty,-2\gamma}\|D\varphi\|_{\infty}(\|y,y^{'}\|_{\mathcal{D}^{2\gamma,\eta}_{w}}/R-\|v,v^{'}\|_{\mathcal{D}^{2\gamma,\eta}_{w}}/R)\\
&\leq C_{\chi,|w|_{\gamma}}\|y-v,(y-v)^{'}\|_{\mathcal{D}^{2\gamma,\eta}_{w}}.
\end{aligned}
\end{equation}
Consequently, we have
\begin{equation}\nonumber
\begin{aligned}
\|f(\chi_{R}(y))-f(\chi_{R}(v))\|_{\infty,-2\gamma}
&\leq C_{\chi,|w|_{\gamma},f}R\|y-v,(y-v)^{'}\|_{\mathcal{D}^{2\gamma,\eta}_{w}}.
\end{aligned}
\end{equation}
Similarly, we have
$$
\|f(\chi_{R}(y))-f(\chi_{R}(v))\|_{\infty,0}\leq C_{\chi,f}R\|y-v,(y-v)^{'}\|_{\mathcal{D}^{2\gamma,\eta}_{w}}.
$$
Finally, according to above estimates, we easily obtain the desired result.
\end{proof}

\begin{lemma}\label{Lemma2.8}
Let $(y,y^{'})$ and $(v,v^{'})\in\mathcal{D}^{2\gamma,\eta}_{w}([0,1];\mathcal{H})$, then there exists a constant $C=C[g,\chi,|w|_{\gamma}]$ such that
\begin{equation}\label{2.25}
\begin{aligned}
\|g_{R}(y)-g_{R}(v),(g_{R}(y)-g_{R}(v))^{'}\|_{\mathscr{D}^{2\gamma,2\gamma,0}_{S,w}}\leq C(R)\|y-v,(y-v)^{'}\|_{\mathcal{D}^{2\gamma,\eta}_{w}}.
\end{aligned}
\end{equation}
\end{lemma}
\begin{proof}
 In the beginning, we give the inequality below which will be used in the following process of proof. Let $g\in \mathcal{C}^{3}_{-2\gamma,0}(\mathcal{H}, \mathcal{H}^{d})$, $x_{1}$, $x_{2}$, $x_{3}$, $x_{4}\in\mathcal{H}_{\theta},\theta\geq-2\gamma$, the estimate as below holds true
\begin{small}
\begin{equation}\label{2.26}
\begin{aligned}
&\|g(x_{1})-g(x_{2})-g(x_{3})+g(x_{4})\|_{\mathcal{H}^{d}_{\theta}}\\
&\leq C_{g}\max\{\|x_{1}\|_{\mathcal{H}_{\theta}},\|x_{2}\|_{\mathcal{H}_{\theta}}\}\|x_{1}-x_{2}-x_{3}+x_{4}\|_{\mathcal{H}_{\theta}}\\
&\quad+C_{g}(\|x_{1}-x_{3}\|_{\mathcal{H}_{\theta}}+\|x_{2}-x_{4}\|_{\mathcal{H}_{\theta}})\|x_{3}-x_{4}\|_{\mathcal{H}_{\theta}}.
\end{aligned}
\end{equation}
\end{small}
The key point of proof is to estimate terms of $\|g(\chi_{R}(y))-g(\chi_{R}(v))\|_{\gamma,-2\gamma}$, $\|(g(\chi_{R}(y))-g(\chi_{R}(v)))^{'}\|_{\gamma,-2\gamma}$ and $|R^{g(\chi_{R}(y))}-R^{g(\chi_{R}(v))}|_{2\gamma,-2\gamma}$. According to the construction of $\chi_{R}(y)$, we easily have the following estimates
\begin{equation}\nonumber
\begin{aligned}
R^{\chi_{R}(y)}_{t,s}&=\hat{\delta}\chi_{R}(y)_{t,s}-S_{ts}\chi_{R}^{'}(y)_{s}\delta w_{t,s}\\
&=\hat{\delta}y_{t,s}\varphi(\|y,y^{'}\|_{\mathcal{D}^{2\gamma,\eta}_{w}}/R)-S_{ts}y^{'}_{s}\varphi(\|y,y^{'}\|_{\mathcal{D}^{2\gamma,\eta}_{w}}/R)\delta w_{t,s}\\
&=R^{y}_{t,s}\varphi(\|y,y^{'}\|_{\mathcal{D}^{2\gamma,\eta}_{w}}/R),
\end{aligned}
\end{equation}
\begin{small}
\begin{equation}\nonumber
\begin{aligned}
\|\chi_{R}(y)\|_{\gamma,-2\gamma}
&\leq\|y\varphi(\|y,y^{'}\|_{\mathcal{D}^{2\gamma,\eta}_{w}}/R)\|_{\gamma,-2\gamma}\leq\varphi(\|y,y^{'}\|_{\mathcal{D}^{2\gamma,\eta}_{w}}/R)\|y\|_{\gamma,-2\gamma}\\
&\leq(1+|w|_{\gamma})\|y,y^{'}\|_{\mathcal{D}^{2\gamma,\eta}_{w}}\leq C_{|w|_{\gamma}}R,
\end{aligned}
\end{equation}
\end{small}
\begin{small}
\begin{equation}\nonumber
\begin{aligned}
\|\chi_{R}(y)\|_{\infty,0}
&=\|y\varphi(\|y,y^{'}\|_{\mathcal{D}^{2\gamma,\eta}_{w}}/R)\|_{\infty,0}=\varphi(\|y,y^{'}\|_{\mathcal{D}^{2\gamma,\eta}_{w}}/R)\|y\|_{\infty,0}\\
&\leq\|y\|_{\infty,0}\leq\|y,y^{'}\|_{\mathcal{D}^{2\gamma,\eta}_{w}}\leq R,
\end{aligned}
\end{equation}
\end{small}
\begin{small}
\begin{equation}\nonumber
\begin{aligned}
|\chi_{R}(y)-\chi_{R}(v)|_{\gamma,-2\gamma}&\leq\varphi(\|y,y^{'}\|_{\mathcal{D}^{2\gamma,\eta}_{w}}/R)|y-v|_{\gamma,-2\gamma}\\
&\quad+\|v\|_{\gamma,-2\gamma}\|D\varphi\|_{\infty}(\|y,y^{'}\|_{\mathcal{D}^{2\gamma,\eta}_{w}}/R-\|v,v^{'}\|_{\mathcal{D}^{2\gamma,\eta}_{w}}/R)\\
&\leq C_{|w|_{\gamma},\chi}\|y-v,(y-v)^{'}\|_{\mathcal{D}^{2\gamma,\eta}_{w}},
\end{aligned}
\end{equation}
\end{small}
\begin{small}
\begin{equation}\nonumber
\begin{aligned}
|\chi_{R}^{'}(y)-\chi_{R}^{'}(v)|_{\gamma,-2\gamma}&=|y^{'}\varphi(\|y,y^{'}\|_{\mathcal{D}^{2\gamma,\eta}_{w}}/R)-v^{'}\varphi(\|v,v^{'}\|_{\mathcal{D}^{2\gamma,\eta}_{w}}/R)|_{\gamma,-2\gamma}\\
&\leq\varphi(\|y,y^{'}\|_{\mathcal{D}^{2\gamma,\eta}_{w}}/R)|y^{'}-v^{'}|_{\gamma,-2\gamma}\\
&\quad+\|v^{'}\|_{\gamma,-2\gamma}\|D\varphi\|_{\infty}(\|y,y^{'}\|_{\mathcal{D}^{2\gamma,\eta}_{w}}/R-\|v,v^{'}\|_{\mathcal{D}^{2\gamma,\eta}_{w}}/R)\\
&\leq C_{|w|_{\gamma},\chi}\|y-v,(y-v)^{'}\|_{\mathcal{D}^{2\gamma,\eta}_{w}},
\end{aligned}
\end{equation}
\end{small}
\begin{small}
\begin{equation}\nonumber
\begin{aligned}
|R^{\chi_{R}(y)}-R^{\chi_{R}(v)}|_{2\gamma,-2\gamma}&=|R^{y}\varphi(\|y,y^{'}\|_{\mathcal{D}^{2\gamma,\eta}_{w}}/R)-R^{v}\varphi(\|v,v^{'}\|_{\mathcal{D}^{2\gamma,\eta}_{w}}/R)|_{2\gamma,-2\gamma}\\
&\leq\varphi(\|y,y^{'}\|_{\mathcal{D}^{2\gamma,\eta}_{w}}/R)|R^{y}-R^{v}|_{2\gamma,-2\gamma}\\
&\quad+|R^{v}|_{2\gamma,-2\gamma}\|D\varphi\|_{\infty}(\|y,y^{'}\|_{\mathcal{D}^{2\gamma,\eta}_{w}}/R-\|v,v^{'}\|_{\mathcal{D}^{2\gamma,\eta}_{w}}/R)\\
&\leq C_{|w|_{\gamma},\chi}\|y-v,(y-v)^{'}\|_{\mathcal{D}^{2\gamma,\eta}_{w}}.
\end{aligned}
\end{equation}
\end{small}

Firstly, applying above estimates and \eqref{2.24}, we have
\begin{small}
\begin{equation}\nonumber
\begin{aligned}
\|g(\chi_{R}(y))-g(\chi_{R}(v))\|_{\gamma,-2\gamma}
\leq |g(\chi_{R}(y))-g(\chi_{R}(v))|_{\gamma,-2\gamma}+\|g(\chi_{R}(y))-g(\chi_{R}(v))\|_{\infty,0},
\end{aligned}
\end{equation}
\begin{equation}\nonumber
\begin{aligned}
&|g(\chi_{R}(y))-g(\chi_{R}(v))|_{\gamma,-2\gamma}\\
&\leq C_{g}\max\{\|\chi_{R}(y)\|_{\infty,-2\gamma},\|\chi_{R}(v)\|_{\infty,-2\gamma}\}|\chi_{R}(y)-\chi_{R}(v)|_{\gamma,-2\gamma}\\
&\quad+C_{g}\left(|\chi_{R}(y)|_{\gamma,-2\gamma}+|\chi_{R}(v)|_{\gamma,-2\gamma}\right)\|\chi_{R}(y)-\chi_{R}(v)\|_{\infty,-2\gamma}\\
&\leq C_{g,|w|_{\gamma},\chi}R(\|\chi_{R}(y)-\chi_{R}(v)\|_{\gamma,-2\gamma}+\|\chi_{R}(y)-\chi_{R}(v)\|_{\infty,0})\\
&\quad+C_{g}\left(\|\chi_{R}(y)\|_{\gamma,-2\gamma}+\|\chi_{R}(y)\|_{\infty,0}+\|\chi_{R}(v)\|_{\gamma,-2\gamma}+\|\chi_{R}(v)\|_{\infty,0}\right)\|\chi_{R}(y)-\chi_{R}(v)\|_{\infty,-2\gamma}\\
&\leq C_{g,|w|_{\gamma},\chi}R\|y-v,(y-v)^{'}\|_{\mathcal{D}^{2\gamma,\eta}_{w}}
\end{aligned}
\end{equation}
\end{small}
and
\begin{small}
\begin{equation}\nonumber
\begin{aligned}
\|g(\chi_{R}(y))-g(\chi_{R}(v))\|_{\infty,0}&\leq C_{g}\max\{\|\chi_{R}(y)\|_{\infty,0},\|\chi_{R}(v)\|_{\infty,0}\}\|\chi_{R}(y)-\chi_{R}(v)\|_{\infty,0}\\
&\leq C_{g,\chi}R\|\chi_{R}(y)-\chi_{R}(v)\|_{\infty,0},
\end{aligned}
\end{equation}
\end{small}
hence, base on above estimates we obtain
\begin{small}
\begin{equation}\nonumber
\begin{aligned}
\|g(\chi_{R}(y))-g(\chi_{R}(v))\|_{\gamma,-2\gamma}&\leq C_{g,|w|_{\gamma},\chi}R\|y-v,(y-v)^{'}\|_{\mathcal{D}^{2\gamma,\eta}_{w}}.
\end{aligned}
\end{equation}
\end{small}

Secondly, since
\begin{small}
\begin{equation}\nonumber
\begin{aligned}
\|Dg(\chi_{R}(y))\chi^{'}_{R}(y)-Dg(\chi_{R}(v))\chi^{'}_{R}(v)\|_{\gamma,-2\gamma}&\leq |Dg(\chi_{R}(y))\chi^{'}_{R}(y)-Dg(\chi_{R}(v))\chi^{'}_{R}(v)|_{\gamma,-2\gamma}\\
&\quad+\|Dg(\chi_{R}(y))\chi^{'}_{R}(y)-Dg(\chi_{R}(v))\chi^{'}_{R}(v)\|_{\infty,0},
\end{aligned}
\end{equation}
\end{small}
\begin{small}
\begin{equation}\nonumber
\begin{aligned}
&|Dg(\chi_{R}(y))\chi^{'}_{R}(y)-Dg(\chi_{R}(v))\chi^{'}_{R}(v)|_{\gamma,-2\gamma}\\
&\leq|Dg(\chi_{R}(y))(\chi^{'}_{R}(y)-\chi^{'}_{R}(v))|_{\gamma,-2\gamma}+|(Dg(\chi_{R}(y))-(Dg(\chi_{R}(v)))\chi^{'}_{R}(v)|_{\gamma,-2\gamma}\\
&\leq\|Dg(\chi_{R}(y))\|_{\infty,\mathcal{L}(\mathcal{H}_{-2\gamma}\otimes\mathbb{R}^{d},\mathcal{H}_{-2\gamma})}|(\chi^{'}_{R}(y)-\chi^{'}_{R}(v))|_{\gamma,-2\gamma}\\
&\quad+\|Dg(\chi_{R}(y))\|_{\gamma,\mathcal{L}(\mathcal{H}_{-2\gamma}\otimes\mathbb{R}^{d},\mathcal{H}_{-2\gamma})}|(\chi^{'}_{R}(y)-\chi^{'}_{R}(v))|_{\infty,-2\gamma}\\
&\quad+|Dg(\chi_{R}(y))-Dg(\chi_{R}(v))|_{-2\gamma,\mathcal{L}(\mathcal{H}_{-2\gamma}\otimes\mathbb{R}^{d},\mathcal{H}_{-2\gamma})}\|\chi^{'}_{R}(v)\|_{\infty,-2\gamma}\\
&\quad+|Dg(\chi_{R}(y))-Dg(\chi_{R}(v))|_{\infty,\mathcal{L}(\mathcal{H}_{-2\gamma}\otimes\mathbb{R}^{d},\mathcal{H}_{-2\gamma})}\|\chi^{'}_{R}(v)\|_{\gamma,-2\gamma}\\
&\leq C_{g}(\|\chi_{R}(y)\|_{\infty,-2\gamma}|\chi^{'}_{R}(y)-\chi^{'}_{R}(v)|_{\gamma,-2\gamma}+|\chi_{R}(y)|_{\gamma,-2\gamma}\|\chi^{'}_{R}(y)-\chi^{'}_{R}(v)\|_{\infty,-2\gamma})\\
&\quad+ C_{g}(\|\chi^{'}_{R}(v)\|_{\infty,-2\gamma}|\chi_{R}(y)-\chi_{R}(v)|_{\gamma,-2\gamma}+|\chi^{'}_{R}(v)|_{\gamma,-2\gamma}\|\chi_{R}(y)-\chi_{R}(v)\|_{\infty,-2\gamma})\\
&\leq C_{g,|w|_{\gamma},\chi}R(\|\chi^{'}_{R}(y)-\chi^{'}_{R}(v)\|_{\gamma,-2\gamma}+\|\chi^{'}_{R}(y)-\chi^{'}_{R}(v)\|_{\infty,0}+\|\chi^{'}_{R}(y)-\chi^{'}_{R}(v)\|_{\infty,-2\gamma})\\
&\quad+C_{g,\chi}R(\|\chi_{R}(y)-\chi_{R}(v)\|_{\gamma,-2\gamma}+\|\chi_{R}(y)-\chi_{R}(v)\|_{\infty,0}+\|\chi_{R}(y)-\chi_{R}(v)\|_{\infty,-2\gamma})
\end{aligned}
\end{equation}
\end{small}
and
\begin{small}
\begin{equation}\nonumber
\begin{aligned}
\|Dg(\chi_{R}(y))\chi^{'}_{R}(y)-Dg(\chi_{R}(v))\chi^{'}_{R}(v)\|_{\infty,0}
&\leq\|Dg(\chi_{R}(y))(\chi^{'}_{R}(y)-\chi^{'}_{R}(v))\|_{\infty,0}\\
&\quad+\|(Dg(\chi_{R}(y))-Dg(\chi_{R}(v)))\chi^{'}_{R}(v)\|_{\infty,0}\\
&\leq C_{g}(\|\chi_{R}(y)\|_{\infty,0}\|\chi^{'}_{R}(y)-\chi^{'}_{R}(v)\|_{\infty,0}\\
&\quad+\|\chi_{R}(y)-\chi_{R}(v))\|_{\infty,0}\|\chi^{'}_{R}(v)\|_{\infty,0})\\
&\leq C_{g,|w|_{\gamma},\chi}R(\|\chi^{'}_{R}(y)-\chi^{'}_{R}(v)\|_{\infty,0}\\
&\quad+\|\chi_{R}(y)-\chi_{R}(v))\|_{\infty,0}).
\end{aligned}
\end{equation}
\end{small}
Using above estimates we easily obtain
\begin{small}
\begin{equation}\nonumber
\begin{aligned}
\|Dg(\chi_{R}(y))\chi^{'}_{R}(y)-Dg(\chi_{R}(v))\chi^{'}_{R}(v)\|_{\gamma,-2\gamma}
\leq C_{g,|w|_{\gamma},\chi}R\|y-v,(y-v)^{'}\|_{\mathcal{D}^{2\gamma,\eta}_{w}}.
\end{aligned}
\end{equation}
\end{small}
For the remainder term of $g(\chi_{R}(y))-g(\chi_{R}(v))$, using \eqref{2.13} we have
\begin{equation}\nonumber
\begin{aligned}
&R^{g(\chi_{R}(y))}_{t,s}-R^{g(\chi_{R}(v))}_{t,s}\\
&=g(\chi_{R}(y_{t}))-g(\chi_{R}(y_{s}))-Dg(\chi_{R}(y_{s}))\delta \chi_{R}(y)_{t,s}-(g(\chi_{R}(v_{t}))-g(\chi_{R}(v_{s}))\\
&\quad-Dg(\chi_{R}(v_{s}))\delta \chi_{R}(v)_{t,s})+Dg(\chi_{R}(y))_{s}R^{\chi_{R}(y)}_{t,s}-Dg(\chi_{R}(v))_{s}R^{\chi_{R}(v)}_{t,s}\\
&\quad-(S_{ts}-I)(g(\chi_{R}(y_{s}))-g(\chi_{R}(v_{s})))\\
&\quad+Dg(\chi_{R}(y_{s}))(S_{ts}-I)\chi_{R}^{'}(y)_{s}\delta w_{t,s}-(Dg(\chi_{R}(v_{s}))(S_{ts}-I)\chi_{R}^{'}(v)_{s}\delta w_{t,s})\\
&\quad-(S_{ts}-I)(Dg(\chi_{R}(y_{s}))\chi_{R}^{'}(y)_{s}\delta w_{t,s}-Dg(\chi_{R}(v)_{s})\chi_{R}^{'}(v)_{s}\delta w_{t,s})\\
&\quad+Dg(\chi_{R}(y_{s}))(S_{ts}-I)\chi_{R}(y)_{s}-Dg(\chi_{R}(v_{s}))(S_{ts}-I)\chi_{R}(v)_{s}\\
&=\romannumeral1+\romannumeral2+\romannumeral3
+\romannumeral4+\romannumeral5+\romannumeral6.
\end{aligned}
\end{equation}
For $\romannumeral1$, using (44) of \cite{MR4284415} twice, we have
\begin{small}
\begin{equation}
\begin{aligned}
&\|\romannumeral1\|_{\mathcal{H}^{d}_{-2\gamma}}\nonumber\\
&=\|\int^{1}_{0}\int^{1}_{0}C_{g}[\tau r^{2}(\chi_{R}(y_{t})-\chi_{R}(v_{t})+(r-\tau r^{2})(\chi_{R}(y_{s})-\chi_{R}(v_{s}))]d\tau dr(\delta\chi_{R}(y)_{t,s})^{2}\|_{\mathcal{H}_{-2\gamma}}\nonumber\\
\end{aligned}
\end{equation}
\end{small}
\begin{small}
\begin{equation}
\begin{aligned}
&\quad+\|\int^{1}_{0}\int^{1}_{0}C_{g}[\tau r^{2}\chi_{R}(v_{t})+(r-\tau r^{2})\chi_{R}(v_{s})]d\tau dr\left((\delta \chi_{R}(y)_{t,s})^{2}-(\delta \chi_{R}(v)_{t,s})^{2}\right)\|_{\mathcal{H}_{-2\gamma}}\nonumber\\
&\leq C_{g}\|\chi_{R}(y)-\chi_{R}(v)\|_{\infty,-2\gamma}\|(\hat{\delta }\chi_{R}(y)_{t,s}+(S_{ts}-I)\chi_{R}(y_{s}))^{2}\|_{\mathcal{H}_{-2\gamma}}\nonumber\\
&\quad+C_{g}\|\chi_{R}(v)\|_{\infty,-2\gamma}\|\hat{\delta }\chi_{R}(y)_{t,s}+(S_{ts}-I)\chi_{R}(y_{s})+\hat{\delta}\chi_{R}(v)_{t,s}+(S_{ts}-I)\chi_{R}(v_{s})\|_{\mathcal{H}_{-2\gamma}}\nonumber\\
&\quad\cdot\|\hat{\delta}(\chi_{R}(y)-\chi_{R}(v))_{t,s}+(S_{ts}-I)(\chi_{R}(y)-\chi_{R}(v))_{s}\|_{\mathcal{H}_{-2\gamma}}\nonumber\\
&\leq C_{g}\|\chi_{R}(y)-\chi_{R}(v)\|_{\infty,-2\gamma}(\|\chi_{R}(y)\|_{\gamma,-2\gamma}|t-s|^{\gamma}+\|\chi_{R}(y)\|_{\infty,0}|t-s|^{2\gamma})^{2}\nonumber\\
&\quad+C_{g}\|\chi_{R}(v)\|_{\infty,-2\gamma}(\|\chi_{R}(y)\|_{\gamma,-2\gamma}|t-s|^{\gamma}+\|\chi_{R}(y)\|_{\infty,0}|t-s|^{2\gamma}+\|\chi_{R}(v)\|_{\gamma,-2\gamma}|t-s|^{\gamma}\\
&\quad+\|\chi_{R}(v)\|_{\infty,0}|t-s|^{2\gamma})\cdot(\|\chi_{R}(y)-\chi_{R}(v)\|_{\gamma,-2\gamma}|t-s|^{\gamma}+\|\chi_{R}(y)-\chi_{R}(v)\|_{\infty,0}|t-s|^{2\gamma}),
\end{aligned}
\end{equation}
\end{small}
hence,
\begin{small}
\begin{equation}\nonumber
\begin{aligned}
\|\romannumeral1\|_{2\gamma,-2\gamma}\leq C_{g,|w|_{\gamma},\chi}R^{2}\|y-v,(y-v)^{'}\|_{\mathcal{D}^{2\gamma,\eta}_{w}}.
\end{aligned}
\end{equation}
\end{small}
For $\romannumeral2$,
\begin{small}
\begin{equation}\nonumber
\begin{aligned}
\|\romannumeral2\|_{\mathcal{H}^{d}_{-2\gamma}}
&\leq\|Dg(\chi_{R}(y_{s}))R^{\chi_{R}(y)}_{t,s}-Dg(\chi_{R}(y_{s}))R^{\chi_{R}(v)}_{t,s}+Dg(\chi_{R}(y_{s}))R^{\chi_{R}(v)}_{t,s}\\
&\quad-Dg(\chi_{R}(v_{s}))R^{\chi_{R}(v)}_{t,s}\|_{\mathcal{H}^{d}_{-2\gamma}}\nonumber\\
&\leq\|Dg(\chi_{R}(y_{s}))(R^{\chi_{R}(y)}_{t,s}-R^{\chi_{R}(v)}_{t,s})\|_{\mathcal{H}^{d}_{-2\gamma}}+\|(Dg(\chi_{R}(y_{s}))-Dg(\chi_{R}(v_{s}))R^{\chi_{R}(v)}_{t,s}\|_{\mathcal{H}^{d}_{-2\gamma}}\nonumber\\
&\leq C_{g}\|\chi_{R}(y)\|_{\infty,-2\gamma}|R^{\chi_{R}(y)}-R^{\chi_{R}(v)}|_{2\gamma,-2\gamma}|t-s|^{2\gamma}\\
&\quad+C_{g}\|\chi_{R}(y)-\chi_{R}(v)\|_{\infty,-2\gamma}|R^{\chi_{R}(v)}|_{2\gamma,-2\gamma}|t-s|^{2\gamma},
\end{aligned}
\end{equation}
\end{small}
hence,
\begin{small}
\begin{equation}\nonumber
\begin{aligned}
\|\romannumeral2\|_{2\gamma,-2\gamma}\leq C_{g,|w|_{\gamma},\chi}R\|y-v,(y-v)^{'}\|_{\mathcal{D}^{2\gamma,\eta}_{w}}.
\end{aligned}
\end{equation}
\end{small}
For $\romannumeral3$, we easily have
\begin{small}
\begin{equation}\nonumber
\begin{aligned}
\|\romannumeral3\|_{\mathcal{H}^{d}_{-2\gamma}}&\leq \|Dg(\chi_{R}(y_{s}))(\chi_{R}(y_{s})-\chi_{R}(v_{s}))\|_{\mathcal{H}^{d}}|t-s|^{2\gamma}\\
&\leq C_{g}\|\chi_{R}(y)\|_{\infty,0}\|\chi_{R}(y)-\chi_{R}(v)\|_{\infty,0}|t-s|^{2\gamma}\\
&\leq C_{g,|w|_{\gamma},\chi}R\|\chi_{R}(y)-\chi_{R}(v)\|_{\infty,0}|t-s|^{2\gamma}
\end{aligned}
\end{equation}
\end{small}
hence,
\begin{small}
\begin{equation}\nonumber
\begin{aligned}
\|\romannumeral3\|_{2\gamma,-2\gamma}\leq C_{g,|w|_{\gamma},\chi}R\|y-v,(y-v)^{'}\|_{\mathcal{D}^{2\gamma,\eta}_{w}}.
\end{aligned}
\end{equation}
\end{small}
For $\romannumeral4$, we have
\begin{small}
\begin{equation}
\begin{aligned}
\|\romannumeral4\|_{\mathcal{H}^{d}_{-2\gamma}}
&\leq\|(Dg(\chi_{R}(y_{s}))-Dg(\chi_{R}(v_{s})))(S_{ts}-I)\chi_{R}^{'}(y)_{s}\delta w_{t,s}\|_{\mathcal{H}^{d}_{-2\gamma}}\nonumber\\
&\quad+\|Dg(\chi_{R}(v_{s}))(S_{ts}-I)(\chi_{R}^{'}(y)_{s}-\chi_{R}^{'}(v)_{s})\delta w_{t,s}\|_{\mathcal{H}^{d}_{-2\gamma}}\nonumber\\
&\leq C_{g}\|\chi_{R}(y)-\chi_{R}(v)\|_{\infty,-2\gamma}\|\chi_{R}^{'}(y)\|_{\infty,0}|w|_{\gamma}|t-s|^{3\gamma}\\
&\quad+C_{g}\|\chi_{R}(y)\|_{\infty,-2\gamma}\|\chi_{R}^{'}(y)-\chi_{R}^{'}(v)\|_{\infty,0}|w|_{\gamma}|t-s|^{3\gamma}\\
&\leq C_{g,|w|_{\gamma},\chi}R\|\chi_{R}(y)-\chi_{R}(v)\|_{\infty,-2\gamma}|w|_{\gamma}|t-s|^{3\gamma}\\
&\quad+C_{g,|w|_{\gamma},\chi}R\|\chi_{R}^{'}(y)-\chi_{R}^{'}(v)\|_{\infty,0}|w|_{\gamma}|t-s|^{3\gamma},
\end{aligned}
\end{equation}
\end{small}
hence,
\begin{small}
\begin{equation}\nonumber
\begin{aligned}
\|\romannumeral4\|_{2\gamma,-2\gamma}\leq C_{g,|w|_{\gamma},\chi}R\|y-v,(y-v)^{'}\|_{\mathcal{D}^{2\gamma,\eta}_{w}}.
\end{aligned}
\end{equation}
\end{small}
For $\romannumeral5$, we have
\begin{small}
\begin{equation}
\begin{aligned}
\|\romannumeral5\|_{\mathcal{H}^{d}_{-2\gamma}}
&\leq\|(S_{ts}-I)(Dg(\chi_{R}(y_{s}))-Dg(\chi_{R}(v_{s})))\chi_{R}^{'}(y)_{s}\delta w_{t,s}\|_{\mathcal{H}^{d}_{-2\gamma}}\nonumber\\
&\quad+\|(S_{ts}-I)Dg(\chi_{R}(v_{s}))(\chi_{R}^{'}(y)_{s}-\chi_{R}^{'}(v)_{s})\delta w_{t,s}\|_{\mathcal{H}^{d}_{-2\gamma}}\nonumber\\
&\leq C_{g}\|\chi_{R}(y)-\chi_{R}(v)\|_{\infty,0}\|\chi_{R}^{'}(y)\|_{\infty,0}|w|_{\gamma}|t-s|^{3\gamma}\\
&\quad+C_{g}\|\chi_{R}(y)\|_{\infty,0}\|\chi_{R}^{'}(y)-\chi_{R}^{'}(v)\|_{\infty,0}|w|_{\gamma}|t-s|^{3\gamma}\\
&\leq C_{g,\chi}R\|\chi_{R}(y)-\chi_{R}(v)\|_{\infty,0}|w|_{\gamma}|t-s|^{3\gamma}\\
&\quad+C_{g,\chi}R\|\chi_{R}^{'}(y)-\chi_{R}^{'}(v)\|_{\infty,0}|w|_{\gamma}|t-s|^{3\gamma},
\end{aligned}
\end{equation}
\end{small}
hence,
\begin{equation}\nonumber
\begin{aligned}
\|\romannumeral5\|_{2\gamma,-2\gamma}\leq C_{g,|w|_{\gamma},\chi}R\|y-v,(y-v)^{'}\|_{\mathcal{D}^{2\gamma,\eta}_{w}}.
\end{aligned}
\end{equation}
For $\romannumeral6$, we have
\begin{small}
\begin{equation}
\begin{aligned}
\|\romannumeral6\|_{\mathcal{H}^{d}_{-2\gamma}}
&\leq\|(Dg(\chi_{R}(y_{s}))-Dg(\chi_{R}(v_{s})))(S_{ts}-I)\chi_{R}(y)_{s}\|_{\mathcal{H}^{d}_{-2\gamma}}\nonumber\\
&\quad+\|Dg(\chi_{R}(v_{s}))(S_{ts}-I)(\chi_{R}(y)_{s}-\chi_{R}(v)_{s})\|_{\mathcal{H}^{d}_{-2\gamma}}\nonumber\\
\end{aligned}
\end{equation}
\end{small}
\begin{small}
\begin{equation}
\begin{aligned}\nonumber
&\leq C_{g}\|\chi_{R}(y)-\chi_{R}(v)\|_{\infty,-2\gamma}\|\chi_{R}(y)\|_{\infty,0}|t-s|^{2\gamma}\\
&\quad+C_{g}\|\chi_{R}(y)\|_{\infty,-2\gamma}\|\chi_{R}(y)-\chi_{R}(v)\|_{\infty,0}|t-s|^{2\gamma}\\
&\leq C_{g,\chi}R\|\chi_{R}(y)-\chi_{R}(v)\|_{\infty,-2\gamma}|t-s|^{2\gamma}\\
&\quad+C_{g,|w|_{\gamma},\chi}R\|\chi_{R}^{'}(y)-\chi_{R}^{'}(v)\|_{\infty,0}|t-s|^{2\gamma},
\end{aligned}
\end{equation}
\end{small}
hence,
\begin{equation}\nonumber
\begin{aligned}
\|\romannumeral6\|_{2\gamma,-2\gamma}\leq C_{g,|w|_{\gamma},\chi}R\|y-v,(y-v)^{'}\|_{\mathcal{D}^{2\gamma,\eta}_{w}}.
\end{aligned}
\end{equation}
Consequently, according to above estimates, we obtain
\begin{small}
\begin{equation}\nonumber
\begin{aligned}
|R^{g(\chi_{R}(y))}-R^{g(\chi_{R}(v))}|_{2\gamma,-2\gamma}&\leq C_{g,|w|_{\gamma},\chi}(R+R^{2})\|y-v,(y-v)^{'}\|_{\mathcal{D}^{2\gamma,\eta}_{w}}\\
&\leq C_{g,|w|_{\gamma},\chi}(R)\|y-v,(y-v)^{'}\|_{\mathcal{D}^{2\gamma,\eta}_{w}}.
\end{aligned}
\end{equation}
\end{small}

Finally, one can easily obtain \eqref{2.25}.
\end{proof}
According to above Lemmas, we will prove that the modified equation \eqref{2.2} obtained by replacing $f$ and $g$ with $f_{R}$ and $g_{R}$ has
a unique solution. To this aim, for $(y,y^{'})\in\mathcal{D}^{2\gamma,\eta}_{w}([0,1];\mathcal{H})$ and $t\in[0,1]$, we introduce
\begin{equation}\label{2.27}
\begin{aligned}
\mathcal{T}_{R}(w,y,y^{'})[t]:=\int^{t}_{0}S_{t u}f_{R}(y_{u})du+\int^{t}_{0}S_{tu}g_{R}(y_{u})d\mathbf{w}_{u},
\end{aligned}
\end{equation}
with mildly Gubinelli derivative $\mathcal{T}_{R}(w,y,y^{'})^{'}=g_{R}(y)$. Base on the estimates derived in the previous lemmas, we have the following result.
\begin{theorem}\label{theorem4}
The mapping
$\mathcal{T}_{R}:\mathcal{D}^{2\gamma,\eta}_{w}([0,1];\mathcal{H})\rightarrow\mathcal{D}^{2\gamma,\eta}_{w}([0,1];\mathcal{H})$,
$$(\mathcal{T}_{R}(w,y,y^{'})[\cdot],\mathcal{T}_{R}(w,y,y^{'})^{'}[\cdot]):=\left(\int^{\cdot}_{0}S_{\cdot u}f_{R}(y_{u})du+\int^{\cdot}_{0}S_{\cdot u}g_{R}(y_{u})d\mathbf{w}_{u},g_{R}(y_{\cdot})\right)$$
has a fixed-point.
\end{theorem}
\begin{proof}
The proof relies on Banach's fixed-point theorem and is similar to the one of \cite{MR4040992} Theorem 4.1 and \cite{MR4284415} Theorem 2.16. In our evolution setting, let $(y,y^{'})$ and $(v,v^{'})\in\mathcal{D}^{2\gamma,\eta}_{w}([0,1];\mathcal{H})$ with $y_{0}=v_{0}$. Base on \eqref{2.17} and \eqref{2.23}, we have that
\begin{small}
\begin{equation}
\begin{aligned}
&\left\|\int^{\cdot}_{0}S_{\cdot u}(f_{R}(y_{u})-f_{R}(v_{u}))du,0\right\|_{\mathcal{D}^{2\gamma,\eta}_{w}}\nonumber\\
&\leq C\left(\|f_{R}(y)-f_{R}(v)\|_{\infty}+\|f_{R}(y)-f_{R}(v)\|_{\infty,-2\gamma}\right)\leq CR\|y-v,y^{'}-v^{'}\|_{\mathcal{D}^{2\gamma,\eta}_{w}}.
\end{aligned}
\end{equation}
\end{small}
Using Lemma \ref{Lemma2.4} and \eqref{2.25}, we obtain
\begin{small}
\begin{equation}
\begin{aligned}
&\left\|\int^{\cdot}_{0}S_{\cdot u}(g_{R}(y_{u})-g_{R}(v_{u}))d\mathbf{w}_{u},g_{R}(y)-g_{R}(v)\right\|_{\mathcal{D}^{2\gamma,\eta}_{w}}\nonumber\\
&\leq C(1+|w|_{\gamma}+|w^{2}|_{2\gamma})(1+|w|_{\gamma})^{2}\|g_{R}(y)-g_{R}(v),(g_{R}(y)-g_{R}(v))^{'}\|_{\mathscr{D}^{2\gamma,2\gamma,0}_{S,w}}\nonumber\\
&\leq C(1+|w|_{\gamma}+|w^{2}|_{2\gamma})(1+|w|_{\gamma})^{2}C(R)\|y-v,(y-v)^{'}\|_{\mathcal{D}^{2\gamma,\eta}_{w}}.
\end{aligned}
\end{equation}
\end{small}
according to above estimates, we derive
\begin{small}
\begin{equation}\label{2.28}
\begin{aligned}
&\left\|\int^{\cdot}_{0}S_{\cdot u}(f_{R}(y_{u})-f_{R}(v_{u}))du+\int^{\cdot}_{0}S_{\cdot u}(g_{R}(y_{u})-g_{R}(v_{u}))d\mathbf{w}_{u},g_{R}(y)-g_{R}(v)
\right\|_{\mathcal{D}^{2\gamma,\eta}_{w}}\\
&\leq(C_{f,\chi,|w|_{\gamma}}R+C_{g,\chi,|w|_{\gamma}}C(R)(1+|w|_{\gamma}+|w^{2}|_{2\gamma})(1+|w|_{\gamma})^{2})\|y-v,(y-v)^{'}\|_{\mathcal{D}^{2\gamma,\eta}_{w}}.
\end{aligned}
\end{equation}
\end{small}
Letting $v\equiv0$ and using $f_{R}(0)=g_{R}(0)=0$, we see from the previous analyze that $\mathcal{T}_{R}$ maps
$\mathcal{D}^{2\gamma,\eta}_{w}([0,1];\mathcal{H})$ into itself. Furthermore, by choosing $R$ small enough we obtain that $\mathcal{T}_{R}$ is a contraction. Consequently, due to Banach's fixed-point theorem, this mapping has a unique fixed-point, i.e. there
exists a unique $(y,y^{'})\in\mathcal{D}^{2\gamma,\eta}_{w}([0,1];\mathcal{H})$ such that $(\mathcal{T}_{R}(\cdot,y,y^{'}),\mathcal{T}_{R}(\cdot,y,y^{'})^{'})=(y,y^{'})$.
\end{proof}

Return to our consideration, in order to decrease the Lipschitz constants of $f$ and $g$ by $\chi_{R}$, the next aim is to characterize $R$ as required.  As
already seen we have to choose $R$ as small as possible. Since in our discussions, it is always required that $R\leq1$ and $C(R)$ is strictly increasing in $R$.
As is often encountered in the theory of stochastic dynamical systems \cite{MR4284415}, since all the estimates depend on the random input, it is meaningful to employ a cut-off technique for a random variable, i.e. $R = R(w)$.
Such an argument will also be used here as follows.

We fix $K>0$ and regard to \eqref{2.28}, set $\tilde{R}(w)$ be the unique solution of

\begin{equation}\label{2.29}
\begin{aligned}
C_{f,\chi,|w|_{\gamma}}\tilde{R}(w)+C_{g,\chi,|w|_{\gamma}}(\tilde{R}(w))(1+|w|_{\gamma}+|w^{2}|_{2\gamma})(1+|w|_{\gamma})^{2}=K
\end{aligned}
\end{equation}
and set
\begin{equation}\label{2.30}
\begin{aligned}
R(w):=\min\{\tilde{R}(w),1\}.
\end{aligned}
\end{equation}
This means that if $R(w)=1$, we apply the cut-off procedure for $\|y,y^{'}\|_{\mathcal{D}^{2\gamma,\eta}_{w}}\leq1/2$ or else if $R(w) < 1$
for $\|y,y^{'}\|_{\mathcal{D}^{2\gamma,\eta}_{w}}\leq R(w)/2$.

In the following sections, we work with a modified version of \eqref{2.2}, where the
drift and diffusion coefficients $f$ and $g$ are replaced by $f_{R(w)}$ and $g_{R(w)}$. For notational simplicity,
the $w$-dependence of $R$ will be dropped whenever there is no confusion.

According to \eqref{2.29}, we have
\begin{lemma}\label{Lemma2.9}
Let $(y,y^{'})$ and $(v,v^{'})\in\mathcal{D}^{2\gamma,\eta}_{w}([0,1];\mathcal{H})$, we have
\begin{equation}\label{2.31}
\begin{aligned}
&\|\mathcal{T}_{R}(w,y,y^{'})-\mathcal{T}_{R}(w,v,v^{'}),(\mathcal{T}_{R}(w,y,y^{'})-\mathcal{T}_{R}(w,v,v^{'}))^{'}\|_{\mathcal{D}^{2\gamma,\eta}_{w}}\\
&\leq K\|y-v,(y-v)^{'}\|_{\mathcal{D}^{2\gamma,\eta}_{w}}.
\end{aligned}
\end{equation}
\end{lemma}

\section{ Random Dynamical system}
In this section we will analyze the dynamics of REEs \eqref{2.2}. Firstly we recall some basic concepts and results on the random dynamical systems theory (\cite{MR1723992}, \cite{MR3624539}), which allow us to study invariant manifolds for \eqref{2.2}.
\begin{definition}\label{definition5}
 Let $(\Omega, \mathcal{F}, \mathbb{P})$ stand for a probability space and $\theta:R\times\Omega\rightarrow\Omega$ be a family of $\mathbb{P}$-preserving
transformations (i.e., $\theta_{t}\mathbb{P} = \mathbb{P}$ for $t\in \mathbb{R}$) having following properties:
\begin{itemize}
 \item the mapping $(t,\omega)\mapsto\theta_{t}\omega$ is $(\mathcal{B}(\mathbb{R})\otimes \mathcal{F},\mathcal{F})$-measurable, where $\mathcal{B}(\cdot)$ denotes the Borel sigma-algebra;
 \item  $\theta_{0}= Id_{\Omega};$
 \item $\theta_{t+s} = \theta_{t}\circ\theta_{s}$ for all $t, s\in\mathbb{ R}.$
 \end{itemize}
Then the quadruple $(\Omega,\mathcal{F}, \mathbb{P},(\theta_{t})_{t\in \mathbb{R}})$ is called a metric dynamical system.
\end{definition}

%Rough paths and rough drivers are usually defined on compact intervals, according to \cite{MR3624539} and \cite{MR11111}, we say $\mathbf{W}=(W,\mathbb{W})\in\mathscr{C}^{\gamma}(\mathbb{R};\mathbb{R}^{d})$ is a $\gamma$-H\"{o}lder rough path if $\mathbf{W}|_{I}\in\mathscr{C}^{\gamma}(I;\mathbb{R}^{d})$ for every compact interval $I\subseteq\mathbb{R}$ containing 0.
In our evolution setting, the construction of metric dynamical system depends on the construction of shift map $\Theta$. According to \cite{MR4284415} we know that shifts act quite naturally on rough paths. For an $\gamma$-H\"{o}lder rough path
$\mathbf{w}=(w,w^{2})$ and $t,\tau\in\mathbb{R}$, let us define the time-shift $\Theta_{\tau}\mathbf{w}=(\theta_{\tau}w,\tilde{\theta}_{\tau}w^{2})$ by
$$\qquad\theta_{\tau}w_{t}:=w_{t+\tau}-w_{\tau},$$
$$\tilde{\theta}_{\tau}w^{2}_{t,s}:=w^{2}_{t+\tau,s+\tau}.$$
Note that $\delta(\theta_{\tau}w)_{t,s}=w_{t+\tau}-w_{s+\tau}$. Furthermore, the shift leaves the path space invariant:
\begin{lemma}[\cite{MR4284415}]\label{lemma31}
Let $T_{1},T_{2},\tau\in\mathbb{R}$, and $\mathbf{w}=(w,w^{2})$ be an $\gamma$-H\"{o}lder rough path on $[T_{1}, T_{2}]$ for $\gamma\in(\frac{1}{3},\frac{1}{2}]$. Then the time-shift $\Theta_{\tau}\mathbf{w}=(\theta_{\tau}w,\tilde{\theta}_{\tau}w^{2})$ is also an $\gamma$-H\"{o}lder rough path on $[T_{1}-\tau, T_{2}-\tau]$.
\end{lemma}

According to \cite{MR3624539}, we consider the follwing concept:
\begin{definition}[\cite{MR3624539}]\label{definition6}
Let $(\Omega,\mathcal{F}, \mathbb{P},(\theta_{t})_{t\in \mathbb{R}})$ be a metric dynamical system. We call $\mathbf{w}=(w,w^{2})$ a rough path cocycle if the identity
 $$\mathbf{w}_{t+s,s}(\omega)=\mathbf{w}_{t,0}(\theta_{s}\omega)$$
holds true for every $\omega\in\Omega$, $s\in \mathbb{R}$ and $t>0$.
\end{definition}
The previous definitions imply that one can use a space of paths as a probability space $\Omega$. As example 3.5 in \cite{MR4284415}, fractional Brownian motion $\mathbf{B^{H}}=(B^{H},\mathbb{B^{H}})$ represents a rough path cocycle, by the same construction of path-space  $(\Omega_{B^{H}},\mathcal{F}_{B^{H}}, \mathbb{P}_{B^{H}})$ of fractional Brownian motion(for further details see \cite{MR3624539}), we have the abstract definition of  metric dynamical systems for our problem modelling the underlying rough
driving process. Now we also need to define the dynamical system structure of the solution operators of our rough evolution equations \eqref{2.2}. Meanwhile, we recall the classical definition of random dynamical system \cite{MR1723992}.

\begin{definition}\label{definition7}
A random dynamical system $\varphi$ on $\mathcal{H}$ over a metric dynamical system $(\Omega,\mathcal{F}, \mathbb{P},(\theta_{t})_{t\in \mathbb{R}})$
is a measurable mapping
$$\varphi:[0,\infty)\times\Omega\times\mathcal{H}\rightarrow\mathcal{H},\quad(t,\omega,x)\mapsto\varphi(t,\omega,x),$$
such that:
\begin{itemize}
 \item $\varphi(0,\omega,\cdot)=Id_{x}$ for all $\omega\in\Omega$;
 \item  $\varphi(t+\tau,\omega,x)=\varphi(t,\theta_{\tau}\omega,\varphi(\tau,\omega,x))$, for all $x\in\mathcal{H}$, $t,\tau\in[0,\infty)$, $\omega\in\Omega$;
 \item $\varphi(t,\omega,\cdot):\mathcal{H}\rightarrow\mathcal{H}$ is continuous for all $t\in[0,\infty)$ and all $\omega\in\Omega$.
 \end{itemize}
\end{definition}

 Now one can hope that the solution operators of \eqref{2.2} generate random dynamical systems. As for all we know, the rough integral given in \eqref{2.6} is pathwise, no exceptional sets occur. For completeness, we give a proof of this fact, see also \cite{MR4284415}.

\begin{lemma}\label{lemma32}
Let $\mathbf{w}$ be a rough path cocycle, then the solution operator
$$t\mapsto\varphi(t,w,\xi)=y_{t}=S_{t}\xi+\int^{t}_{0}S_{tu}f(y_{u})du+\int^{t}_{0}S_{tu}g(y_{u})d\mathbf{w}_{u},$$
for any $t\in[0,\infty)$ of the REE \eqref{2.2} generates a random dynamical system over the metric dynamical
system $(\Omega_{w},\mathcal{F}_{w},\mathbb{P},(\theta_{t})_{t\in \mathbb{R}})$.
\end{lemma}
\begin{proof}
The proof is analogous to \cite{MR11112} and \cite{MR4284415} Lemma 3.7. The difficultly is to check the cocycle property for the solution operator. Here we just prove the cocycle property. Firstly, we easily check that if $(y,y^{'})\in\mathcal{D}^{2\gamma,\eta}_{w}([T_{1}+\tau,T_{2}+\tau];\mathcal{H})$ then $(y_{\cdot+\tau},y^{'}_{\cdot+\tau})\in\mathcal{D}^{2\gamma,\eta}_{\theta_{\tau}w}([T_{1},T_{2}];\mathcal{H})$, here $T_{1},T_{2}\in\mathbb{R}$ with $T_{1}<T_{2}$. The $\gamma$-H\"{o}lder continuity of $y_{\cdot+\tau}$ and $y^{'}_{\cdot+\tau}$ is obvious. For the remainder we have
\begin{equation}
\begin{aligned}
\|R^{y_{\cdot+\tau}}_{t,s}\|_{\mathcal{H}_{-2\gamma}}&=\|\hat{\delta}y_{t+\tau,s+\tau}-S_{ts}y^{'}_{s+\tau}\delta w_{t+\tau,s+\tau}\|_{\mathcal{H}_{-2\gamma}}\nonumber\\
&=\|\hat{\delta}y_{t+\tau,s+\tau}-S_{(t+\tau)-(s+\tau)}y^{'}_{s+\tau}\delta w_{t+\tau,s+\tau}\|_{\mathcal{H}_{-2\gamma}}\\
&=\|R^{y}_{t+\tau,s+\tau}\|_{\mathcal{H}_{-2\gamma}}\leq|R^{y}|_{2\gamma,-2\gamma}|t-s|^{2\gamma}.
\end{aligned}
\end{equation}
Next, we will obtain the shift property of rough integral.  Let $\mathcal{P}$ be a partition of $[\tau,t+\tau]$, then we have
\begin{small}
\begin{equation}\label{3.1}
\begin{aligned}
&\int^{t+\tau}_{\tau}S_{t+\tau-u}g(y_{u})d\mathbf{w}_{u}\\
&=\lim\limits_{|\mathcal{P}|\rightarrow0}\sum\limits_{[u,v]\in\mathcal{P}}\left(S_{t+\tau-u}g(y_{u})\delta
w_{v,u}+S_{t+\tau-u}Dg(y_{u})y^{'}_{u}w^{2}_{v,u}\right)\\
&=\lim\limits_{|\mathcal{P}^{'}|\rightarrow0}\sum\limits_{[u^{'},v^{'}]\in\mathcal{P}^{'}}\left(S_{t-u^{'}}g(y_{u^{'}+\tau})\delta w_{v^{'}+\tau,u^{'}+\tau}+
  S_{t-u^{'}}Dg(y_{u^{'}+\tau})y^{'}_{u^{'}+\tau}w^{2}_{v^{'}+\tau,u^{'}+\tau}\right)\\
&=\lim\limits_{|\mathcal{P}^{'}|\rightarrow0}\sum\limits_{[u^{'},v^{'}]\in\mathcal{P}^{'}}\left(S_{t-u^{'}}g(y_{u^{'}+\tau})\delta(\theta_{\tau}w)_{v^{'},u^{'}}+
  S_{t-u^{'}}Dg(y_{u^{'}+\tau})y^{'}_{u^{'}+\tau}\tilde{\theta}_{\tau}w^{2}_{v^{'},u^{'}}\right)\\
&=\int^{t}_{0}S_{t-u^{'}}g(y_{u^{'}+\tau})d\Theta_{\tau}\mathbf{w}_{u^{'}},
\end{aligned}
\end{equation}
\end{small}
where $\mathcal{P}^{'}$ is a partition of $[0,t]$ given by $\mathcal{P}^{'}:=\{[s-\tau,t-\tau]:[s,t]\in\mathcal{P}\}$.
The proof of the cocycle property and mensurability of solution operators is similar to \cite{MR11112} and \cite{MR4284415}, here we omit.
%We calculate
%\begin{small}
%\begin{equation}
%\begin{aligned}
%y_{t+\tau}&=S_{t+\tau}\xi+\int^{t+\tau}_{0}S_{t+\tau-u}f(y_{u})du+\int^{t+\tau}_{0}S_{t+\tau-u}g(y_{u})d\mathbf{w}_{u}\nonumber\\
%&=S_{t}S_{\tau}\xi+\int^{\tau}_{0}S_{t+\tau-u}f(y_{u})du+\int^{t+\tau}_{\tau}S_{t+\tau-u}f(y_{u})du\\
%&\quad+\int^{\tau}_{0}S_{t+\tau-u}g(y_{u})d\mathbf{w}_{u}+\int^{t+\tau}_{\tau}S_{t+\tau-u}g(y_{u})d\mathbf{w}_{u}\\
%&=S_{t}\left(S_{\tau}\xi+\int^{\tau}_{0}S_{\tau-u}f(y_{u})du+\int^{\tau}_{0}S_{\tau-u}g(y_{u})d\mathbf{w}_{u}\right)\\
%&\quad+\int^{t}_{0}S_{t-u}f(y_{u+\tau})du+\int^{t}_{0}S_{t-u}g(y_{u+\tau})d\Theta_{\tau}\mathbf{w}_{u}\\
%&=S_{t}y_{\tau}+\int^{t}_{0}S_{t-u}f(y_{u+\tau})du+\int^{t}_{0}S_{t-u}g(y_{u+\tau})d\Theta_{\tau}\mathbf{w}_{u},
%\end{aligned}
%\end{equation}
%\end{small}
\end{proof}
The next concept of tempered random
variables \cite{MR1723992} is of fundamental importance in the study of local random invariant manifolds.
\begin{definition}\label{definition8}
A random variable $\tilde{R}:\Omega\rightarrow(0,\infty)$ is called tempered from above, with respect to a
metric dynamical system $(\Omega,\mathcal{F}, \mathbb{P},(\theta_{t})_{t\in \mathbb{R}})$, if
\begin{equation}\label{3.2}
\begin{aligned}
\limsup_{t\rightarrow\pm\infty}\frac{\ln^{+}\tilde{R}(\theta_{t}\omega)}{|t|}=0,\quad\mbox{for}\quad\mbox{all}\quad \omega\in\Omega,
\end{aligned}
\end{equation}

where $\ln^{+}a:=\max\{\ln a,0\}$. A random variable is called tempered from below if $1/\tilde{R}$ is tempered from
above. A random variable is tempered if and only if is tempered from above and from below.
\end{definition}
The temperedness reflexes the subexponential growth of the mapping $t\rightarrow\tilde{R}(\theta_{t}\omega)$, according to \cite{MR1723992} Proposition 4.1.3, a sufficient condition for temperedness is
\begin{equation}\label{3.3}
\begin{aligned}
\mathbb{E}\sup_{t\in[0,1]}\tilde{R}(\theta_{t}\omega)<\infty.
\end{aligned}
\end{equation}
Moreover, if the random variable $\tilde{R}$ is tempered from below with $t\rightarrow\tilde{R}(\theta_{t}\omega)$ continuous for all $\omega\in\Omega$,
then for every $\varepsilon>0$ there exists a constant $C[\varepsilon,\omega]>0$ such that
\begin{equation}\label{3.4}
\begin{aligned}
\tilde{R}(\theta_{t}\omega)\geq C[\varepsilon,\omega]e^{-\varepsilon|t|},
\end{aligned}
\end{equation}
for any $\omega\in\Omega$.

According to \cite{MR4284415} Lemma 3.9 and Lemma 3.10, we can assume that $\mathbf{w}=(w,w^{2})$ is a rough path cocycle such that the random variables $$R_{1}(w)=|w|_{\gamma}\quad\mbox{and}\quad R_{2}(w^{2})=|w^{2}|_{2\gamma}$$
are tempered from above. These will be necessary for the proof of existence for a local unstable manifold. One needs to ensure that for initial conditions belonging to a ball with a sufficiently small tempered from below radius, the corresponding trajectories remain within such a ball( for further details, refer to \cite{MR2110052}, \cite{MR3289240}, \cite{MR2593602}, \cite{MR4284415}). By previous discussions, we easily obtain the result below:

\begin{lemma}\label{lemma33}
The random variable $R(w)$ in \eqref{2.30} is tempered from below.
\end{lemma}

\section{Local unstable manifolds for REEs }
In this section, we will study the existence of local unstable manifolds for (2.2) by the Lyapunov-Perron method which is similar to the one employed in \cite{MR2593602}, \cite{MR11112} and \cite{MR4284415}. However, here we want to connect the theory of random invariant manifolds for REEs as in (\cite{MR3376184}, \cite{MR2110052}, \cite{MR2593602}, \cite{MR11112}) to rough paths theory.

Firstly, as in \cite{MR2593602} and \cite{MR3289240}, we assume that the spectrum $\sigma(A)$ of linear operator $A$ only consists of a countable number of eigenvalues, and it splits as
\begin{equation}\label{4.1}
\begin{aligned}
\sigma(A)=\{\lambda_{k},k\in\mathbb{N}\}=\sigma_{u}\bigcup\sigma_{s},
\end{aligned}
\end{equation}
with both $\sigma_{u}$ and $\sigma_{s}$ nonempty, and
$$\sigma_{u}\subset\{z\in\mathbb{C}:Re z>0\}\quad \mbox{and}\quad \sigma_{s}\subset\{z\in\mathbb{C}:Re z<0\},$$
where $\mathbb{C}$ denotes the set of complex numbers and $\sigma_{u}=\{\lambda_{k},\cdot\cdot\cdot\lambda_{N}\}$ for some $N>0$.
Denote the corresponding eigenvectors for $\{\lambda_{k},k\in\mathbb{N}\}$ by $\{e_{1},\cdot\cdot\cdot,e_{N},e_{N+1},\cdot\cdot\cdot\}$,
furthermore, assume that the eigenvectors form an orthonormal basis of $\mathcal{H}$. Thus there is an invariant orthogonal decomposition $\mathcal{H}=\mathcal{H}_{u}\oplus\mathcal{H}_{s}$ with dim$\mathcal{H}_{u}=N$, such that for the restrictions which are $A_{u}=A|_{\mathcal{H}_{u}}$, $A_{s}=A|_{\mathcal{H}_{s}}$, one has $\sigma_{u}=\{z\in\sigma(A_{u})\}$ and $\sigma_{s}=\{z\in\sigma(A_{s})\}$. Moreover, $e^{A_{u}t}$ is a group of linear operators on $\mathcal{H}_{u}$, and there
exist projections $\pi^{u}$ and $\pi^{s}$, such that $\pi^{u}+\pi^{s}=Id_{\mathcal{H}}$, $A_{u}=\pi^{u}A$ and $A_{s}=\pi^{s}A$.
Furthermore, we assume that the projections $\pi^{u}$ and $\pi^{s}$ commute with $A$. Additionally, suppose that there are constants $0<\beta<\alpha$ such that
\begin{equation}\label{4.2}
\begin{aligned}
\|e^{tA_{u}}x\|\leq e^{\alpha t}\|x\|,\quad t\leq0,
\end{aligned}
\end{equation}
\begin{equation}\label{4.3}
\begin{aligned}
\|e^{tA_{s}}x\|\leq e^{-\beta t}\|x\|,\quad t\geq0.
\end{aligned}
\end{equation}

\begin{definition}\label{definition9}
 If a random set $\mathcal{M}^{u}(w)$, which is invariant respect to random dynamical system $\varphi$ (i.e. $\varphi(t,w,\mathcal{M}^{u}(w))\subset\mathcal{M}^{u}(\theta_{t}w)$) for $t\in\mathbb{R}$ and $w\in\Omega_{w}$, can be represented as
\begin{equation}\label{4.4}
\begin{aligned}
\mathcal{M}^{u}(w)=\{\xi+h^{u}(\xi,w):\xi\in\mathcal{H}^{u}\},
\end{aligned}
\end{equation}
where $h^{u}(\xi,W):\mathcal{H}^{u}\rightarrow\mathcal{H}^{s}$. Then we call $\mathcal{M}^{u}(w)$ an unstable manifold. Moreover, $h^{u}(0,w)=0$ and $\mathcal{M}^{u}(w)$ is tangent to $\mathcal{H}^{u}$ at the origin, meaning that the tangency condition $Dh^{u}(0,w)=0$ is satisfied.
\end{definition}

%We will prove the existence of a local center manifold $\mathcal{M}^{c}_{loc}(W)$ for (2.2), namely (4.4) holds true when $\xi$ belongs to a random ball of $\mathcal{H}^{c}$ with a tempered radius. The Lipschitz continuity of $h^{c}$ with respect to $\xi$ will also be justified.

Here, we employ the Lyapunov-Perron method which is similar with \cite{MR11112} and \cite{MR4284415}. As well, in our case, the continuous-time Lyapunov-Perron
map for \eqref{2.2} is presented by 
\begin{equation}\label{4.5}
\begin{aligned}
J(w,y)[\tau]&:=S^{u}_{\tau}\xi^{u}+\int^{\tau}_{0}S^{u}_{\tau u}\pi^{u}f(y_{u})du+\int^{\tau}_{0}S^{u}_{\tau u}\pi^{u}g(y_{u})d\mathbf{w}_{u}\\
              &\quad+\int^{\tau}_{-\infty}S^{s}_{\tau u}\pi^{s}f(y_{u})du+\int^{\tau}_{-\infty}S^{s}_{\tau u}\pi^{s}g(y_{u})d\mathbf{w}_{u}
\end{aligned}
\end{equation}
for $\tau\leq0$.
  Thanks to the presence of the rough integral we couldn't directly deal with \eqref{4.5}, we need to track $|w|_{\gamma}$ and $|w^{2}|_{\gamma}$ that appear in \eqref{2.14} on a finite-time horizon. Similar to \cite{MR2593602} and \cite{MR11112}, we derive an appropriate discretized Lyapunov-Perron map and prove that it has a fixed-point in a suitable function space. The local unstable manifold will be developed for the discrete-time random dynamical system and will be shown that it holds true for the original continuous-time one, as in \cite{MR2593602}.

Analogous to \cite{MR2593602}, \cite{MR11112} and \cite{MR4284415}, we only need to deal with rough integral on time-interval $[0,1]$. Let $w\in\Omega_{w}$, $t\in[0,1]$ and $i\in\mathbb{Z}^{-}$, replacing $\tau$ by $t+i-1$ in \eqref{4.5}, we have
\begin{small}
\begin{equation}\label{4.6}
\begin{aligned}
&J(w,y)[t+i-1]\\
 &=S^{u}_{t+i-1}\xi^{u}\\
 &\quad-\sum^{i+1}_{k=0}S^{u}_{t+i-1-k}\left(\int^{1}_{0}S^{u}_{1-u}\pi^{u}f(y_{u+k-1})du+\int^{1}_{0}S^{u}_{1-u}\pi^{u}g(y_{u+k-1})d\Theta_{k-1}\mathbf{w}_{u}\right)\\
 &\quad-\int^{1}_{t}S^{u}_{t-u}\pi^{u}f(y_{u+i-1})du-\int^{1}_{t}S^{u}_{t-u}\pi^{u}g(y_{u+i-1})d\Theta_{i-1}\mathbf{w}_{u}\\
 &\quad+\sum^{i-1}_{k=-\infty}S^{s}_{t+i-1-k}\left(\int^{1}_{0}S^{s}_{1-u}\pi^{s}f(y_{u+k-1})du+\int^{1}_{0}S^{s}_{1-u}\pi^{s}g(y_{u+k-1})d\Theta_{k-1}\mathbf{w}_{u}\right)\\
 &\quad+\int^{t}_{0}S^{s}_{t-u}\pi^{s}f(y_{u+i-1})du+\int^{t}_{0}S^{s}_{t-u}\pi^{s}g(y_{u+i-1})d\Theta_{i-1}\mathbf{w}_{u},
\end{aligned}
\end{equation}
\end{small}
by applying \eqref{4.6}, we will give the structure of the discrete Lyapunov-Perron map.
for $(y,y^{'})\in\mathcal{D}^{2\gamma,\eta}_{w}([0,1];\mathcal{H})$, we denote
\begin{equation}\label{4.7}
\begin{aligned}
\mathcal{T}^{s/u}(w,y,y^{'})[\cdot]=\int^{\cdot}_{0}S^{s/u}_{t-u}\pi^{s/u}f(y_{u})du+\int^{\cdot}_{0}S^{s/u}_{t-u}\pi^{s/u}g(y_{u})d\mathbf{w}_{u},
\end{aligned}
\end{equation}
\begin{equation}\label{4.8}
\begin{aligned}
\tilde{\mathcal{T}}^{u}(w,y,y^{'})[\cdot]=\int^{1}_{\cdot}S^{u}_{t-u}\pi^{u}f(y_{u})du+\int^{1}_{\cdot}S^{u}_{t-u}\pi^{u}g(y_{u})d\mathbf{w}_{u},
\end{aligned}
\end{equation}
where $\big(\mathcal{T}^{s/u}(w,y,y^{'})[\cdot]\big)^{'}=\big(\tilde{\mathcal{T}}^{u}(w,y,y^{'})[\cdot]\big)^{'}=g(y_{\cdot})$. Meanwhile, in our evolution setting, we directly deal with solutions of the REEs \eqref{2.2}. It is essential problem that we need to find an appropriate space for the fixed-point argument. For this, similar to \cite{MR11112} and \cite{MR4284415} we introduce the following function space which helps us incorporate the discretized version of \eqref{4.5}.

Let $\delta=\frac{\alpha-\beta}{2}>0$,
we let $BC_{\delta}(\mathcal{D}^{2\gamma,\eta}_{w})$ be the space of a sequence of mildly controlled rough paths $(\mathbf{y},\mathbf{y}^{'}):=(y^{i-1},(y^{i-1})^{'})_{i\in\mathbb{Z}^{-}}$ with $y^{i-1}_{0}=y^{i-2}_{1}$, where $(y^{i-1},(y^{i-1})^{'})\in\mathcal{D}^{2\gamma,\eta}_{w}([0,1],\mathcal{H})$, if
\begin{equation}\label{4.9}
\|\mathbf{y},\mathbf{y}^{'}\|_{BC_{\delta}(\mathcal{D}^{2\gamma,\eta}_{W})}:=\sup_{i\in\mathbb{Z}^{-}}e^{-\delta(i-1)}\|y^{i-1},(y^{i-1})^{'}\|_{\mathcal{D}^{2\gamma,\eta}_{w}([0,1];\mathcal{H})}<\infty.
\end{equation}
In the following, for notational simplicity, we denote $\tilde{y}[i-1,t]=\tilde{y}^{i-1}_{t}$ for $t\in[0,1]$ and $\tilde{y}[\tau]=\tilde{y}[i-1,t]$ if $\tau=t+i-1$.

Next, we modify \eqref{2.2} by the cut-off function given in Section 2, i.e. we replace $f$ and $g$ by $f_{R}$ respectively $g_{R}$.
According to \eqref{4.6}, it is reasonable to introduce the discrete Lyapunov-Perron transform
$J_{d}(w,\mathbf{y},\xi)$ for a sequence of mildly controlled rough paths  as the pair $J_{d}(w,\mathbf{y},\xi):=(J^{1}_{d}(w,\mathbf{y},\xi),J^{2}_{d}(w,\mathbf{y},\xi))$, where $\mathbf{y}\in BC_{\delta}(\mathcal{D}^{2\gamma,\eta}_{w})$ and $\xi\in\mathcal{H}$, the precise structure is given below. The dependence of $J_{d}$ on the cut-off parameter $R$ is indicated by the subscript $R$. For $t\in[0,1]$, $w\in\Omega_{w}$ and $i\in\mathbb{Z}^{-}$, we define
\begin{small}
\begin{equation}\label{4.10}
\begin{aligned}
 &J^{1}_{R,d}(w,\mathbf{y},\xi)[i-1,t]\\
 &=S^{u}_{t+i-1}\xi^{u}
 -\sum^{i+1}_{k=0}S^{u}_{t+i-1-k}\left(\int^{1}_{0}S^{u}_{1-u}\pi^{u}f_{R}(y^{k-1}_{u})du+\int^{1}_{0}S^{u}_{1-u}\pi^{u}g_{R}(y^{k-1}_{u})d\Theta_{k-1}\mathbf{w}_{u}\right)\\
 &\quad-\int^{1}_{t}S^{u}_{t-u}\pi^{u}f_{R}(y^{i-1}_{u})du-\int^{1}_{t}S^{u}_{t-u}\pi^{u}g_{R}(y^{i-1}_{u})d\Theta_{i-1}\mathbf{w}_{u}\\
 &\quad+\sum^{i-1}_{k=-\infty}S^{s}_{t+i-1-k}\left(\int^{1}_{0}S^{s}_{1-u}\pi^{s}f_{R}(y^{k-1}_{u})du+\int^{1}_{0}S^{s}_{1-u}\pi^{s}g_{R}(y^{k-1}_{u})d\Theta_{k-1}\mathbf{w}_{u}\right)\\
 &\quad+\int^{t}_{0}S^{s}_{t-u}\pi^{s}f_{R}(y^{i-1}_{u})du+\int^{t}_{0}S^{s}_{t-u}\pi^{s}g_{R}(y^{i-1}_{u})d\Theta_{i-1}\mathbf{w}_{u},
\end{aligned}
\end{equation}
\end{small}
Moreover, $J^{2}_{R,d}(w,\mathbf{y},\xi)$ is denoted as the mildly Gubinelli derivative of $J^{1}_{R,d}(w,\mathbf{y},\xi)$, i.e. $J^{2}_{R,d}(w,\mathbf{y},\xi)[i-1,t]:=(J^{1}_{R,d}(w,\mathbf{y},\xi)[i-1,t])^{'}$. Notice that one can easily obtain $\xi^{u}=\pi^{u}J^{1}_{R,d}(w,\mathbf{y},\xi)[-1,1]$ by setting $i=0$ and $t=1$.

In the following, we will prove that \eqref{4.10} maps $(\mathbf{y},\mathbf{y}^{'})\in BC_{\delta}(\mathcal{D}^{2\gamma,\eta}_{w})$ into itself and is a contractive mapping if the constant $K$ given in \eqref{2.29} is chosen small enough.
\begin{theorem}\label{theorem14}
In our setting, if $K$ satisfies the gap condition
\begin{equation}\label{4.11}
\begin{aligned}
K\left(\frac{e^{\beta+\delta}(Ce^{-\delta}+1)}{1-e^{-(\beta+\delta)}}+\frac{(e^{-(\alpha-\delta)}-1)(Ce^{-\delta}+e^{\alpha-\delta})}{1-e^{\alpha-\delta}}\right)\leq \frac{1}{2},
\end{aligned}
\end{equation}
then, the map $J_{R,d}:\Omega\times BC_{\delta}(\mathcal{D}^{2\gamma,\eta}_{w})\rightarrow BC_{\delta}(\mathcal{D}^{2\gamma,\eta}_{w})$ possesses a unique fixed-point $\Gamma\in BC_{\delta}(\mathcal{D}^{2\gamma,\eta}_{w})$. Also, the map $\xi^{u}\rightarrow\Gamma(\xi^{u},w)\in BC_{\delta}(\mathcal{D}^{2\gamma,\eta}_{w})$ is Lipschitz continuous.
\end{theorem}
\begin{proof}
Let $(\mathbf{y},\mathbf{y}^{'}):=(y^{i-1},(y^{i-1})^{'})_{i\in\mathbb{Z}^{-}}$ and $(\mathbf{v},\mathbf{v}^{'}):=(v^{i-1},(v^{i-1})^{'})_{i\in\mathbb{Z}^{-}}\in BC_{\delta}(\mathcal{D}^{2\gamma,\eta}_{w})$ with $\pi^{u}y^{-1}=\pi^{u}v^{-1}=\xi^{u}$. Firstly, we give several estimates which is essential for the proof. According to Lemma 2.5, we easily have
\begin{equation}\label{4.12}
\begin{aligned}
\|S^{u}_{\cdot+i-1}\xi^{u},0\|_{BC_{\delta}(\mathcal{D}^{2\gamma,\eta}_{w})}\leq Ce^{(\alpha-\delta)(i-1)}\|\xi^{u},0\|_{\mathcal{D}^{2\gamma,\eta}_{W}}\leq Ce^{(\alpha-\delta)(i-1)}\|\xi^{u}\|,
\end{aligned}
\end{equation}
the above expression keeps bounded for $i\in\mathbb{Z}^{-}$.
Denote
 $$\Lambda=\mathcal{T}^{s}_{R}(\theta_{k-1}w,y^{k-1},(y^{k-1})^{'})[1]-\mathcal{T}^{s}_{R}(\theta_{k-1}w,v^{k-1},(v^{k-1})^{'})[1],$$ from \eqref{2.14}, one has
 $$\|\Lambda\|_{\mathcal{H}}\leq\|y^{k-1}-v^{k-1},(y^{k-1}-v^{k-1})^{'}\|_{\mathcal{D}^{2\gamma,\eta}_{w}}$$
by \eqref{4.12}, we have
\begin{equation}\label{4.13}
\begin{aligned}
\|S^{s}_{\cdot+i-1-k}\Lambda,(S^{s}_{\cdot+i-1-k}\Lambda)^{'}\|_{\mathcal{D}^{2\gamma,\eta}_{w}}&=\|S^{s}_{\cdot+i-1-k}\Lambda,0\|_{\mathcal{D}^{2\gamma,\eta}_{w}}\\
&\leq Ce^{-\beta(i-1-k)}\|\Lambda\|_{\mathcal{H}}\\
&\leq Ce^{-\beta(i-1-k)}\|y^{k-1}-v^{k-1},(y^{k-1}-v^{k-1})^{'}\|_{\mathcal{D}^{2\gamma,\eta}_{w}}.
\end{aligned}
\end{equation}
Similarly, denote
 $$\tilde{\Lambda}=\mathcal{T}^{u}_{R}(\theta_{k-1}w,y^{k-1},(y^{k-1})^{'})[1]-\mathcal{T}^{u}_{R}(\theta_{k-1}w,v^{k-1},(v^{k-1})^{'})[1],$$we easily have
 \begin{equation}\label{4.14}
\begin{aligned}
\|S^{u}_{\cdot+i-1-k}\tilde{\Lambda},(S^{u}_{\cdot+i-1-k}\tilde{\Lambda})^{'}\|_{\mathcal{D}^{2\gamma,\eta}_{w}}
\leq Ce^{\alpha(i-1-k)}\|y^{k-1}-v^{k-1},(y^{k-1}-v^{k-1})^{'}\|_{\mathcal{D}^{2\gamma,\eta}_{w}}.
\end{aligned}
\end{equation}

Next, for the stable part of \eqref{4.10}, due to \eqref{2.31}, \eqref{4.13} and the norm of $BC_{\delta}(\mathcal{D}^{2\gamma,\eta}_{w})$, we have
\begin{small}
\begin{equation}
\begin{aligned}\nonumber
 &\sum^{i-1}_{k=-\infty}e^{-\delta(i-1)}\|S^{s}_{\cdot+i-1-k}(\mathcal{T}^{s}_{R}(\theta_{k-1}w,y^{k-1},(y^{k-1})^{'})[1]-\mathcal{T}^{s}_{R}(\theta_{k-1}w,v^{k-1},(v^{k-1})^{'})[1]),\\
 &\qquad \left(S^{s}_{\cdot+i-1-k}(\mathcal{T}^{s}_{R}(\theta_{k-1}w,y^{k-1},(y^{k-1})^{'})[1]-\mathcal{T}^{s}_{R}(\theta_{k-1}w,v^{k-1},(v^{k-1})^{'})[1])\right)^{'}\|_{\mathcal{D}^{2\gamma,\eta}_{w}}\\
 &\quad+e^{-\delta(i-1)}\|\mathcal{T}^{s}_{R}(\theta_{i-1}w,y^{i-1},(y^{i-1})^{'})[\cdot]-\mathcal{T}^{s}_{R}(\theta_{i-1}w,v^{i-1},(v^{i-1})^{'})[\cdot],\\
 &\qquad\left(\mathcal{T}^{s}_{R}(\theta_{i-1}w,y^{i-1},(y^{i-1})^{'})[\cdot]-\mathcal{T}^{s}_{R}(\theta_{i-1}w,v^{i-1},(v^{i-1})^{'})[\cdot]\right)^{'}\|_{\mathcal{D}^{2\gamma,\eta}_{w}}\\
 &\leq\sum^{i-1}_{k=-\infty}e^{-\delta(i-1)}Ce^{-\beta(i-1-k)}K\|y^{k-1}-v^{k-1},(y^{k-1}-v^{k-1})^{'}\|_{\mathcal{D}^{2\gamma,\eta}_{w}}\\
 &\quad+e^{-\delta(i-1)}K\|y^{i-1}-v^{i-1},(y^{i-1}-v^{i-1})^{'}\|_{\mathcal{D}^{2\gamma,\eta}_{w}}\\
 &\leq\sum^{i-1}_{k=-\infty}e^{-\delta(i-1)}Ce^{-\beta(i-1-k)}e^{\delta(k-1)}Ke^{-\delta(k-1)}\|y^{k-1}-v^{k-1},(y^{k-1}-v^{k-1})^{'}\|_{\mathcal{D}^{2\gamma,\eta}_{w}}\\
 &\quad+e^{-\delta(i-1)}K\|y^{i-1}-v^{i-1},(y^{i-1}-v^{i-1})^{'}\|_{\mathcal{D}^{2\gamma,\eta}_{w}}\\
\end{aligned}
\end{equation}
\end{small}
\begin{small}
\begin{equation}
\begin{aligned}\nonumber
 &\leq\sum^{i-1}_{k=-\infty}e^{-(\beta+\delta)(i-1-k)}Ce^{-\delta}Ke^{-\delta(k-1)}\|y^{k-1}-v^{k-1},(y^{k-1}-v^{k-1})^{'}\|_{\mathcal{D}^{2\gamma,\eta}_{w}}\\
 &\quad+e^{-(\beta+\delta)(i-1-i)}e^{-\delta(i-1)}K\|y^{i-1}-v^{i-1},(y^{i-1}-v^{i-1})^{'}\|_{\mathcal{D}^{2\gamma,\eta}_{w}}\\
 &\leq\sum^{i}_{k=-\infty}e^{-(\beta+\delta)(i-1-k)}K(Ce^{-\delta}+1)e^{-\delta(k-1)}\|y^{k-1}-v^{k-1},(y^{k-1}-v^{k-1})^{'}\|_{\mathcal{D}^{2\gamma,\eta}_{w}}\\
 &\leq \frac{Ke^{\beta+\delta}(Ce^{-\delta}+1)}{1-e^{-(\beta+\delta)}}\|\mathbf{y}-\mathbf{v},(\mathbf{y}-\mathbf{v})^{'}\|_{BC_{\delta}(\mathcal{D}^{2\gamma,\eta}_{w})}.
\end{aligned}
\end{equation}
\end{small}
Similarly, for the unstable part, according to \eqref{2.31}, \eqref{4.14} and the norm of $BC_{\delta}(\mathcal{D}^{2\gamma,\eta}_{w})$, we have
\begin{small}
\begin{equation}
\begin{aligned}\nonumber
 &\sum^{i+1}_{k=0}e^{-\delta(i-1)}\|S^{u}_{\cdot+i-1-k}(\mathcal{T}^{u}_{R}(\theta_{k-1}w,y^{k-1},(y^{k-1})^{'})[1]-\mathcal{T}^{u}_{R}(\theta_{k-1}w,v^{k-1},(v^{k-1})^{'})[1]),\\
 &\qquad\left(S^{u}_{\cdot+i-1-k}(\mathcal{T}^{u}_{R}(\theta_{k-1}w,y^{k-1},(y^{k-1})^{'})[1]-\mathcal{T}^{u}_{R}(\theta_{k-1}w,v^{k-1},(v^{k-1})^{'})[1])\right)^{'}\|_{\mathcal{D}^{2\gamma,\eta}_{w}}\\
 &\quad+e^{-\delta(i-1)}\|\tilde{\mathcal{T}}^{u}_{R}(\theta_{i-1}w,y^{i-1},(y^{i-1})^{'})[\cdot]-\tilde{\mathcal{T}}^{u}_{R}(\theta_{i-1}w,v^{i-1},(v^{i-1})^{'})[\cdot],\\
 &\qquad\left(\tilde{\mathcal{T}}^{u}_{R}(\theta_{i-1}w,y^{i-1},(y^{i-1})^{'})[\cdot]-\tilde{\mathcal{T}}^{u}_{R}(\theta_{i-1}w,v^{i-1},(v^{i-1})^{'})[\cdot]\right)^{'}\|_{\mathcal{D}^{2\gamma,\eta}_{w}}\\
 &\leq\sum^{i+1}_{k=0}e^{-\delta(i-1)}Ce^{\alpha(i-1-k)}K\|y^{k-1}-v^{k-1},(y^{k-1}-v^{k-1})^{'}\|_{\mathcal{D}^{2\gamma,\eta}_{w}}\\
 &\quad+e^{-\delta(i-1)}K\|y^{i-1}-v^{i-1},(y^{i-1}-v^{i-1})^{'}\|_{\mathcal{D}^{2\gamma,\eta}_{w}}\\
 &\leq\sum^{i+1}_{k=0}e^{-\delta(i-1)}Ce^{\alpha(i-1-k)}e^{\delta(k-1)}Ke^{-\delta(k-1)}\|y^{k-1}-v^{k-1},(y^{k-1}-v^{k-1})^{'}\|_{\mathcal{D}^{2\gamma,\eta}_{w}}\\
 &\quad+e^{-\delta(i-1)}K\|y^{i-1}-v^{i-1},(y^{i-1}-v^{i-1})^{'}\|_{\mathcal{D}^{2\gamma,\eta}_{w}}\\
  \end{aligned}
\end{equation}
\end{small}
\begin{small}
\begin{equation}
\begin{aligned}\nonumber
 &\leq\sum^{i+1}_{k=0}e^{(\alpha-\delta)(i-1-k)}Ce^{-\delta}Ke^{-\delta(k-1)}\|y^{k-1}-v^{k-1},(y^{k-1}-v^{k-1})^{'}\|_{\mathcal{D}^{2\gamma,\eta}_{w}}\\
 &\quad+e^{(\alpha-\delta)(i-1-i)}e^{\alpha-\delta}e^{-\delta(i-1)}K\|y^{i-1}-v^{i-1},(y^{i-1}-v^{i-1})^{'}\|_{\mathcal{D}^{2\gamma,\eta}_{w}}\\
 &\leq\sum^{i}_{k=0}e^{(\alpha-\delta)(i-1-k)}K(Ce^{-\delta}+e^{\alpha-\delta})e^{-\delta(k-1)}\|y^{k-1}-v^{k-1},(y^{k-1}-v^{k-1})^{'}\|_{\mathcal{D}^{2\gamma,\eta}_{w}}\\
 &\leq\frac{K(e^{-(\alpha-\delta)}-1)(Ce^{-\delta}+e^{\alpha-\delta})}{1-e^{\alpha-\delta}}\|\mathbf{y}-\mathbf{v},(\mathbf{y}-\mathbf{v})^{'}\|_{BC_{\delta}(\mathcal{D}^{2\gamma,\eta}_{w})}.
\end{aligned}
\end{equation}
\end{small}
Combining previous estimates, we obtain that
$$\|J_{R,d}(w,\mathbf{y},\xi)-J_{R,d}(w,\mathbf{v},\xi)\|_{BC_{\delta}(\mathcal{D}^{2\gamma,\eta}_{w})}\leq\frac{1}{2}\|\mathbf{y}-\mathbf{v},(\mathbf{y}-\mathbf{v})^{'}\|_{BC_{\delta}(\mathcal{D}^{2\gamma,\eta}_{w})}.$$
When $\mathbf{v}\equiv0$, we easily have that $J_{R,d}$ maps $BC_{\delta}(\mathcal{D}^{2\gamma,\eta}_{w})$ into itself.
Applying fixed-point argument, we deduce that $J_{R,d}(w,\mathbf{y},\xi^{u})$ has a unique fixed-point $\Gamma(\xi^{u},w)\in BC_{\delta}(\mathcal{D}^{2\gamma,\eta}_{w})$ for every $\xi^{u}\in\mathcal{H}_{u}$, meanwhile, for $\xi^{u}_{1}$, $\xi^{u}_{2}\in\mathcal{H}_{u}$, we have
\begin{equation}
\begin{aligned}
 &\|\Gamma(\xi^{u}_{1},w)-\Gamma(\xi^{u}_{2},w)\|_{BC_{\delta}(\mathcal{D}^{2\gamma,\eta}_{w})}\\
 &\qquad=\|J_{R,d}(w,\Gamma(\xi^{u}_{1},w),\xi^{u}_{1})-J_{R,d}(w,\Gamma(\xi^{u}_{2},w),
 \xi^{u}_{2})\|_{BC_{\delta}(\mathcal{D}^{2\gamma,\eta}_{w})}\nonumber\\
 &\qquad\leq\|J_{R,d}(w,\Gamma(\xi^{u}_{1},w),\xi^{u}_{1})-J_{R,d}(w,\Gamma(\xi^{u}_{1},w),\xi^{u}_{2})\|_{BC_{\delta}(\mathcal{D}^{2\gamma,\eta}_{w})}\nonumber\\
 &\qquad\quad+\|J_{R,d}(w,\Gamma(\xi^{u}_{1},w),\xi^{u}_{2})-J_{R,d}(w,\Gamma(\xi^{u}_{2},w),\xi^{u}_{2})\|_{BC_{\delta}(\mathcal{D}^{2\gamma,\eta}_{w})}\nonumber\\
 &\qquad\leq\|S^{u}_{\cdot+i-1}(\xi^{u}_{1}-\xi^{u}_{2}),0\|_{BC_{\delta}(\mathcal{D}^{2\gamma,\eta}_{w})}+\frac{1}{2}\|\Gamma(\xi^{u}_{1},w)-\Gamma(\xi^{u}_{2},w)\|_{BC_{\delta}(\mathcal{D}^{2\gamma,\eta}_{w})}\nonumber\\
 &\qquad\leq Ce^{(\alpha-\delta)}\|\xi^{u}_{1}-\xi^{u}_{2}\|+\frac{1}{2}\|\Gamma(\xi^{u}_{1},w)-\Gamma(\xi^{u}_{2},w)\|_{BC_{\delta}(\mathcal{D}^{2\gamma,\eta}_{w})},
\end{aligned}
\end{equation}
which implies that $\Gamma(\xi^{u},w)$ is Lipschitz continuous.
\end{proof}
At last, as similar discussion have taken place in \cite{MR2593602}, \cite{MR11112} and \cite{MR4284415}, we derive a local unstable manifold for our REEs \eqref{2.2}. The proof of following results is identical to the one of \cite{MR11112} and \cite{MR4284415} , we omit here. In the following we denote $B_{\mathcal{H}_{u}}(0,\rho(w))$ as a ball in $\mathcal{H}_{u}$, which is centered at 0 and has a random radius $\rho(w)$.
\begin{lemma}\label{lemma41}
The local unstable manifold of \eqref{2.2} is given by the graph of a Lipschitz function i.e.
\begin{equation}\label{4.15}
\mathcal{M}^{u}_{loc}(w)=\{\xi+h^{u}(\xi,w):\xi\in B_{\mathcal{H}_{u}}(0,\rho(w))\},
\end{equation}
where, $\rho(w)$ is a tempered from below random variable and
$$h^{u}(\xi,w):=\pi^{s}\Gamma(\xi,w)[-1,1]|_{B_{\mathcal{H}_{u}}(0,\rho(w))},$$
that is
\begin{equation}
\begin{aligned}
h^{u}(\xi,w)&=\sum^{0}_{k=-\infty}S^{s}_{-k}\int^{1}_{0}S^{s}_{1-u}\pi^{s}f(\Gamma(\xi,w)[k-1,u])du\nonumber\\
&\quad+\sum^{0}_{k=-\infty}S^{s}_{-k}\int^{1}_{0}S^{s}_{1-u}\pi^{s}g(\Gamma(\xi,w)[k-1,u])d\Theta_{k-1}\mathbf{w}_{u}.
\end{aligned}
\end{equation}

\end{lemma}

 According to previous analysis, we easily obtain:
\begin{theorem}\label{theorem24}
The local unstable manifold of \eqref{2.2} is given by the graph of a Lipschitz function i.e.
$$\mathcal{M}^{u}_{loc}(w)=\{\xi+h^{u}(\xi,w):\xi\in B_{\mathcal{H}_{u}}(0,\hat{\rho}(w)\},$$
where, $\hat{\rho}(w)$ is a tempered frow below random variable and
$$h^{u}(\xi,w):=\int^{0}_{-\infty}S^{s}_{-u}\pi^{s}f(y_{u})du+\int^{0}_{-\infty}S^{s}_{-u}\pi^{s}g(y_{u})d\mathbf{w}_{u}.$$
\end{theorem}
\subsection{Example}
Consider the $2m$-th order parabolic partial equation
\begin{equation}\nonumber
  \left\{
   \begin{aligned}
     dy(u,x)=&\left(L_{2m}y_{u}(x)+\mu y_{u}(x)+f(y_{u}(x))\right)du+g(y_{u}(x))dw_{u}(x),u\in[0,T],\\
     y(0)= &\xi\in\mathcal{O},\\
     \frac{\partial y}{\partial\nu}(u,x)=&0, (u,x)\in(0,T)\times\partial\mathcal{O}, k=0,1,\cdot\cdot\cdot,m-1.
  \end{aligned}
 \right.
 \end{equation}
where $\frac{\partial}{\partial\nu}$ stands for the normal derivative, $\mathcal{O}$ is a bounded domain in $\mathbb{R}^{d}$ with a smooth boundary,
$$-L_{2m}=\sum_{|\kappa|\leq2m}a_{\kappa}(x)D^{\kappa}$$
is a uniformly elliptic operator with $a_{\kappa}\in\mathcal{C}^{\infty}(\bar{\mathcal{O}})$ and $w$ is a $\gamma$ H\"{o}lder continuous path with $1/3<\gamma\leq1/2$.

 We can consider the above  equation as \eqref{2.2} in the space $\mathcal{H}=L^{2}(\mathcal{O})$. Let
 \begin{small}
 $A=L_{2m}+\mu$, $\mbox{Dom}(A)=H^{2m}(\mathcal{O})\cap H^{m}_{0}(\mathcal{O})$ if $2m>\frac{d}{2}$, thus we have that $\mathcal{H}_{-2\gamma}=H^{2m-2\gamma}_{0}(\mathcal{O})$ and the requirement about $2m$ to such that $2m>\frac{d}{2}+4\gamma$ .
 \end{small}
 As we all know that $A$ has a compact resolvent and has countably many eigenvalues $\lambda_{j}$ of finite multiplicity, that tend to $-\infty$ when $j\rightarrow\infty$. In additional, the associated eigenfunctions $\{e_{j}\}_{j\in\mathbb{N}}$ form an orthogonal basis of $\mathcal{H}$. Set $\mu>0$ sufficiently large such that there exists $j^{*}\in\mathbb{N}$
 $$\lambda_{j^{*}+1}\leq-\beta<0<\alpha\leq\lambda_{j^{*}}.$$
 Let $\mathcal{H}_{u}=\mbox{span}(e_{j}:\lambda_{j}\geq\alpha)$ and $\mathcal{H}_{s}$ be its orthogonal complement space in $\mathcal{H}$. i.e. $\mathcal{H}$ has an invariant splitting $\mathcal{H}=\mathcal{H}_{u}\oplus\mathcal{H}_{s}$. Meanwhile, the nonlinear terms $f$ and $g$ satisfy our assumptions in \eqref{2.2}.
\bibliographystyle{abbrv}

\end{document}